\documentclass[12pt,oneside,reqno]{amsart}

\usepackage{amsmath,amssymb,amsthm,textcomp}
\usepackage{amsfonts,graphicx}
\usepackage[mathscr]{eucal}
\usepackage{color}
\usepackage{csquotes}
\usepackage[backend=bibtex,%
firstinits=true,%
doi=false,%
isbn=true,%
url=false,%
maxnames=99]{biblatex}%

\AtEveryBibitem{\clearfield{issn}}
\AtEveryCitekey{\clearfield{issn}}
\numberwithin{equation}{section}
\DeclareNameAlias{sortname}{last-first}
\theoremstyle{definition}
\usepackage{mathtools}
\numberwithin{equation}{section}

\newcommand{\ncom}{\newcommand}

\ncom{\beq}{\begin{equation}}
\ncom{\eeq}{\end{equation}}
\ncom{\bea}{\begin{eqnarray*}}
\ncom{\eea}{\end{eqnarray*}}
\ncom{\beqa}{\begin{eqnarray}}
\ncom{\eeqa}{\end{eqnarray}}
\ncom{\nno}{\nonumber}
\ncom{\non}{\nonumber}
\ncom{\ds}{\displaystyle}
\ncom{\half}{\frac{1}{2}}
\ncom{\mbx}{\makebox{.25cm}}
\ncom{\hs}{\mbox{\hspace{.25cm}}}
\ncom{\rar}{\rightarrow}
\ncom{\Rar}{\Rightarrow}
\ncom{\noin}{\noindent}
\ncom{\bc}{\begin{center}}
\ncom{\ec}{\end{center}}
\ncom{\sz}{\scriptsize}
\ncom{\rf}{\ref}
\ncom{\s}{\sqrt{2}}
\ncom{\sgm}{\sigma}
\ncom{\Sgm}{\Sigma}
\ncom{\psgm}{\sigma^{\prime}}
\ncom{\dt}{\delta}
\ncom{\Dt}{\Delta}
\ncom{\lmd}{\lambda}
\ncom{\Lmd}{\Lambda}
\ncom{\Th}{\Theta}
\ncom{\e}{\eta}
\ncom{\eps}{\epsilon}
\ncom{\pcc}{\stackrel{P}{>}}
\ncom{\lp}{\stackrel{L_{p}}{>}}
\ncom{\dist}{{\rm\,dist}}
\ncom{\sspan}{{\rm\,span}}
\ncom{\re}{{\rm Re\,}}
\ncom{\im}{{\rm Im\,}}
\ncom{\sgn}{{\rm sgn\,}}
\ncom{\ba}{\begin{array}}
\ncom{\ea}{\end{array}}
\ncom{\hone}{\mbox{\hspace{1em}}}
\ncom{\htwo}{\mbox{\hspace{2em}}}
\ncom{\hthree}{\mbox{\hspace{3em}}}
\ncom{\hfour}{\mbox{\hspace{4em}}}
\ncom{\vone}{\vskip 2ex}
\ncom{\vtwo}{\vskip 4ex}
\ncom{\vonee}{\vskip 1.5ex}
\ncom{\vthree}{\vskip 6ex}
\ncom{\vfour}{\vspace*{8ex}}
\ncom{\norm}{\|\;\;\|}
\ncom{\integ}[4]{\int_{#1}^{#2}\,{#3}\,d{#4}}
\ncom{\vspan}[1]{{{\rm\,span}\{ #1 \}}}
\ncom{\dm}[1]{ {\displaystyle{#1} } }
\ncom{\ri}[1]{{#1} \index{#1}}

\newtheorem{theorem}{\bf Theorem}[section]
\newtheorem{remark}{\bf Remark}[section]
\newtheorem{proposition}{Proposition}[section]

\newtheorem{corollary}{Corollary}[section]

\newtheorem{definition}{Definition}[section]
\newtheoremstyle
    {remarkstyle}
    {}
    {11pt}
    {}
    {}
    {\bfseries}
    {:}
    {     }
    {\thmname{#1} \thmnumber{#2} }

\theoremstyle{remarkstyle}

\def\eps{\varepsilon}

\def\E{{\mathbb E}}

\begin{document}
\title{Skellam and time-changed variants of the generalized fractional counting process}
\author[Kuldeep Kumar Kataria]{Kuldeep Kumar Kataria}
\address{Kuldeep Kumar Kataria, Department of Mathematics, Indian Institute of Technology Bhilai, Raipur 492015, India.}
 \email{kuldeepk@iitbhilai.ac.in}
\author[Mostafizar Khandakar]{Mostafizar Khandakar}
\address{Mostafizar Khandakar, Department of Mathematics, Indian Institute of Technology Bhilai, Raipur 492015, India.}
\email{mostafizark@iitbhilai.ac.in}
\date{July 17, 2021}
\begin{abstract}
In this paper, we study a Skellam type variant of the generalized counting process (GCP), namely, the generalized Skellam process. Some of its distributional properties such as the probability mass function, probability generating function, mean, variance and covariance are obtained. Its fractional version, namely, the generalized fractional Skellam process (GFSP) is considered by time-changing it with an independent inverse stable subordinator. It is observed that the GFSP is a Skellam type version of the generalized fractional counting process (GFCP) which is a fractional variant of the GCP. It is shown that the one-dimensional distributions of the GFSP are not infinitely divisible. An integral representation for its state probabilities is obtained. We establish its long-range dependence property by using its variance and covariance structure. Also, we consider two time-changed versions of the GFCP. These are obtained by time-changing the GFCP by an independent L\'evy subordinator and its inverse. Some particular cases of these time-changed processes are discussed by considering specific L\'evy subordinators.
\end{abstract}

\maketitle
\section{Introduction}
The point processes with random time are of particular interest in the theory of stochastic processes due to their potential applications in areas such as finance, hydrology, econometrics, {\it etc.} The time fractional Poisson process (TFPP) and the space fractional Poisson process are two extensively studied fractional extensions of the Poisson process. These time-changed processes are obtained by choosing an independent stable subordinator and its inverse process as a time-change component in the Poisson process (see Laskin (2003), Beghin and Orsingher (2009), Meerschaert {\it et al.} (2011), Orsingher and Polito (2012) {\it etc.})

Di Crescenzo {\it et al.} (2016) introduced and studied the generalized fractional counting process (GFCP) $\{M^{\alpha}(t)\}_{t\ge0}$, $0<\alpha\le 1$ whose state probabilities $p^{\alpha}(n,t)=\mathrm{Pr}\{M^{\alpha}(t)=n\}$ satisfy the following system of fractional differential equations:
\begin{equation}\label{cre}
\partial_{t}^{\alpha}p^{\alpha}(n,t)=-\Lambda p^{\alpha}(n,t)+	\sum_{j=1}^{\min\{n,k\}}\lambda_{j}p^{\alpha}(n-j,t),\ \ n\ge0,
\end{equation}
with the initial conditions
\begin{equation*}
p^{\alpha}(n,0)=\begin{cases}
1,\ \ n=0,\\
0,\ \ n\ge1.
\end{cases}
\end{equation*}
Here, $\Lambda=\lambda_{1}+\lambda_{2}+\dots+\lambda_{k}$ for a fixed positive integer $k$ and $\partial_{t}^{\alpha}$ is the Caputo fractional derivative defined in (\ref{caputo}).  The GFCP performs $k$ kinds of jumps of amplitude $1,2,\dots,k$ with positive rates $\lambda_{1}, \lambda_{2},\dots,\lambda_{k}$, respectively. 

For $\alpha=1$, the GFCP reduces to the generalized counting process (GCP) $\{M(t)\}_{t\ge0}$. For $k=1$, the GFCP and the GCP reduces to the TFPP and the Poisson process, respectively. It is known that (see Di Crescenzo {\it et al.} (2016))
\begin{equation}\label{ed}
M^{\alpha}(t)\stackrel{d}{=}M\left(Y_{\alpha}(t)\right),
\end{equation}
where the GCP $\{M(t)\}_{t\ge0}$ is independent of the inverse stable subordinator $\{Y_{\alpha}(t)\}_{t\ge0}$. Here, $\stackrel{d}{=}$ denotes equality in distribution. Kataria and Khandakar (2021c) studied some additional properties of these generalized processes which includes an application of the GCP in risk theory. It is shown by them that many known counting processes such as the Poisson process of order $k$, the P\'olya-Aeppli process of order $k$, the negative binomial process, the convoluted Poisson process and their fractional versions {\it etc.} are particular cases of the GFCP.

It is important to note that the recently studied time-changed processes are mainly constructed by time-changing a point process with an independent subordinator and its inverse. A subordinator $\{D_f(t)\}_{t\ge0}$ is a one-dimensional L\'evy process  which is characterized by the following Laplace transform (see Applebaum (2009), Section 1.3.2):
\begin{equation*}
\mathbb{E}\left(e^{-sD_f(t)}\right)=e^{-tf(s)},
\end{equation*}
where
\begin{equation*}
f(s)=\int_0^\infty \left(1-e^{-sx}\right)\,\mu(\mathrm{d}x),\ \ s>0,
\end{equation*}
is called the Bern\v stein function. Here, $\mu$ is a non-negative  L\'evy measure that satisfies $\mu([0,\infty))=\infty$ and
$\displaystyle\int_0^\infty \min\{x,1\}\,\mu(\mathrm{d}x)<\infty$. It has non-decreasing sample paths and $D_f(0)=0$ almost surely ({\it a.s.}). The first passage time of a subordinator is called the inverse subordinator. It is defined as
\begin{equation*}
H_f (t)\coloneqq\inf\{r\ge0: D_f (r)> t\}, \ \ t\ge0.
\end{equation*}

In this paper, we first introduce a Skellam type variant of the GCP. We call it the generalized Skellam process (GSP) and denote it by $\{\mathcal{S}(t)\}_{t\ge0}$. It is defined as follows:
\begin{equation*}
\mathcal{S}(t)\coloneqq M_{1}(t)-M_{2}(t),
\end{equation*}
where  $\{M_{1}(t)\}_{t\ge0}$ and $\{M_{2}(t)\}_{t\ge0}$ are independent generalized counting processes with positive rates $\lambda_{j}$'s and $\mu_{j}$'s, $j=1,2,\dots,k$, respectively. For any integer $n$, we obtain its state probability $q(n,t)=\mathrm{Pr}\{\mathcal{S}(t)=n\}$ in the following form:
\begin{equation*}
q(n,t)=e^{-(\Lambda+\bar{\Lambda})t}\left(\Lambda/\bar{\Lambda}\right)^{n/2}I_{|n|}\left(2t\sqrt{\Lambda \bar{\Lambda}}\right),
\end{equation*}
where  ${\Lambda}=\sum_{j=1}^{k}\lambda_{j}$ and $\bar{\Lambda}=\sum_{j=1}^{k}\mu_{j}$. Here, $I_{|n|}(\cdot)$ denotes the modified Bessel function of first kind defined in (\ref{bessel}). Also, we obtain its probability generating function (pgf), characteristic function, mean, variance, covariance  {\it etc}. 

We consider a fractional version of the GSP, namely, the generalized fractional Skellam process (GFSP). It is denoted by $\{\mathcal{S}^{\alpha}(t)\}_{t\ge0}$, $0<\alpha\leq1$. It is defined as the GSP time-changed by an independent inverse stable subordinator, that is,
\begin{equation*}
\mathcal{S}^{\alpha}(t)\coloneqq
\begin{cases*}
\mathcal{S}(Y_{\alpha}(t)),\ \ 0<\alpha<1,\\
\mathcal{S}(t),\ \ \alpha=1.
\end{cases*}
\end{equation*}
The GFSP is a Skellam type version of the GFCP. It is shown that the state probabilities $q^{\alpha}(n,t)=\mathrm{Pr}\{\mathcal{S}^{\alpha}(t)=n\}$, $n\in \mathbb{Z}$ of  GFSP solves the following system of fractional differential equations: 
\begin{equation*}
\partial_{t}^{\alpha}q^{\alpha}(n,t)=\Lambda\left(q^{\alpha}(n-1,t)-q^{\alpha}(n,t)\right)-\bar{\Lambda}\left(q^{\alpha}(n,t)-q^{\alpha}(n+1,t)\right),
\end{equation*}
with $q^{\alpha}(0,0)=1$ and $q^{\alpha}(n,0)=0, \ n\neq0$. We derive its $r$th factorial moment and obtain an integral representation for its probability mass function (pmf). It is shown that the GFSP exhibits the long-range dependence (LRD) property whereas its increment process has the short-range dependence (SRD) property. We have also proved that the one-dimensional distributions of GFSP are not infinitely divisible.

We consider a time-changed version of the GFCP, namely, the time-changed  generalized fractional counting process-I (TCGFCP-I). It is obtained by time-changing the GFCP by an independent L\'evy subordinator $\{D_f (t)\}_{t\ge0}$ such that $\mathbb{E}\left(D^r_f(t)\right)<\infty$ for all $r>0$. We establish a version of the law of iterated logarithm for it. For $\alpha=1$, we show that the TCGFCP-I is equal in distribution to a limiting case of a suitable compound Poisson process. It is shown that the TCGFCP-I exhibits the LRD property under some suitable restrictions on the L\'evy subordinator. Some particular cases of TCGFCP-I are discussed by taking specific L\'evy subordinator such as the gamma subordinator, tempered stable subordinator and inverse Gaussian subordinator. We derive the L\'evy measure and the associated system of governing differential equations for these particular cases. Another time-changed version of the GFCP, namely, the time-changed  generalized  fractional counting process-II (TCGFCP-II) is considered which is obtained by time-changing the GFCP by an independent inverse subordinator. Some particular cases of the TCGFCP-II are discussed.

\section{Preliminaries}
In this section, we give some known definitions and results about fractional derivatives, some special functions, inverse stable subordinator, the GCP and its fractional version. The results presented here will be used later.
\subsection{Fractional derivatives}
The Caputo fractional derivative is defined as (see Kilbas {\it et al.} (2006))
\begin{equation}\label{caputo}
\partial_{t}^{\alpha}f(t):=\left\{
\begin{array}{ll}
\dfrac{1}{\Gamma{(1-\alpha)}}\displaystyle\int^t_{0} (t-s)^{-\alpha}f'(s)\,\mathrm{d}s,\ \ 0<\alpha<1,\\\\
f'(t),\ \ \alpha=1.
\end{array}
\right.
\end{equation}
For $\gamma\geq 0$, the Riemann-Liouville (R-L) fractional derivative is defined as (see Kilbas {\it et al.} (2006))
\begin{equation}\label{RLd}
D_t^{\gamma}f(t):=\left\{
\begin{array}{ll}
\dfrac{1}{\Gamma{(m-\gamma)}}\displaystyle \frac{\mathrm{d}^{m}}{\mathrm{d}t^{m}}\int^t_{0} \frac{f(s)}{(t-s)^{\gamma+1-m}}\,\mathrm{d}s,\ \ m-1<\gamma<m,\\\\
\displaystyle\frac{\mathrm{d}^{m}}{\mathrm{d}t^{m}}f(t),\ \ \gamma=m,
\end{array}
\right.
\end{equation}
where $m$ is a positive integer.

The following relationship holds for the Caputo and R-L fractional derivatives (see Meerschaert and Straka  (2013)):
\begin{equation}\label{relation}
\partial_{t}^{\alpha}f(t)=D_{t}^{\alpha}f(t)-f(0+)\frac{t^{-\alpha}}{\Gamma(1-\alpha)}.
\end{equation}

\subsection{Some special functions}
Here, we briefly describe three special functions.
\subsubsection{Mittag-Leffler function}
The three-parameter Mittag-Leffler function is defined as (see Kilbas {\it et al.} (2006), p. 45)
\begin{equation*}
E_{\alpha,\beta}^{\gamma}(x)\coloneqq\frac{1}{\Gamma(\gamma)}\sum_{j=0}^{\infty} \frac{\Gamma(j+\gamma)x^{j}}{j!\Gamma(j\alpha+\beta)},\ \ x\in\mathbb{R},
\end{equation*}
where $\alpha>0$, $\beta>0$ and $\gamma>0$. 

It reduces to two-parameter Mittag-Leffler function for $\gamma=1$. It further reduces to the Mittag-Leffler function for  $\gamma=\beta=1$. 

The following result holds for $n\ge0$ (see Kilbas {\it et al.} (2006), Eq. (1.8.22)):
\begin{equation}\label{2.2desww}
E_{\beta, \gamma}^{(n)}(x)=n!E_{\beta, n\beta+\gamma}^{n+1}(x),
\end{equation}
where $E_{\beta, \gamma}^{(n)}(\cdot)$ denotes the $n$th derivative of two-parameter Mittag-Leffler function.
\subsubsection{Bessel function}
The modified Bessel function of first kind is defined as (see Sneddon (1956), p. 114):
\begin{equation}\label{bessel}
I_{n}(x)=\sum_{j=0}^{\infty}\frac{(x/2)^{2j+n}}{j!(j+n)!},\  \ x\in\mathbb{R}.
\end{equation}
For any integer $n$, the following properties hold (see Sneddon (1956) , p. 115):
\begin{equation}\label{besseli}
\begin{cases*}
I_{n}(x)=I_{-n}(x),\\
\dfrac{\mathrm{d}}{\mathrm{d}x}I_{n}(x)=\dfrac{1}{2}\left(I_{n-1}(x)+I_{n+1}(x)\right).
\end{cases*}
\end{equation} 
\subsubsection{Wright function}
The Wright function is defined as (see Mainardi (2010)):
\begin{equation*}
M_{\alpha}(x)=\sum_{j=0}^{\infty}\frac{(-x)^{j}}{j!\Gamma(1-j\alpha-\alpha)},\  \ x\in\mathbb{R},
\end{equation*}
where $0<\alpha<1$.
\subsection{Inverse stable subordinator}
A stable subordinator $\{D_{\alpha}(t)\}_{t\ge0}$, $0<\alpha<1$, is a non-decreasing L\'evy process. Its Laplace transform is given by $\mathbb{E}\left(e^{-sD_{\alpha}(t)}\right)=e^{-ts^{\alpha}}$, $s>0$. Its first passage time $\{Y_{\alpha}(t)\}_{t\ge0}$ is called the inverse stable subordinator which is defined as 
\begin{equation*}
Y_{\alpha}(t)\coloneqq \inf\{x\ge0:D_{\alpha}(x)>t\}.
\end{equation*}
The mean, variance and covariance of the inverse stable subordinator are given by (see Leonenko {\it et al.} (2014))
\begin{align}\label{meani}
\mathbb{E}\left(Y_{\alpha}(t)\right)&=\frac{t^{\alpha}}{\Gamma(\alpha+1)},\\
\operatorname{ Var}\left(Y_{\alpha}(t)\right)&=\left(\frac{2}{\Gamma(2\alpha+1)}-\frac{1}{\Gamma^{2}(\alpha+1)}\right)t^{2\alpha},\label{xswe331}\\
\operatorname{Cov}\left(Y_{\alpha}(s),Y_{\alpha}(t)\right)&=\frac{1}{\Gamma^2(\alpha+1)}\left( \alpha s^{2\alpha}B(\alpha,\alpha+1)+F(\alpha;s,t)\right),\ \ 0<s\le t \label{covin},
\end{align}
where $F(\alpha;s,t)=\alpha t^{2\alpha}B(\alpha,\alpha+1;s/t)-(ts)^{\alpha}$. Here, $B(\alpha,\alpha+1)$  and $B(\alpha,\alpha+1;s/t)$ denote the beta function and the incomplete beta function, respectively. 
  
For fixed $s$ and large $t$, the following result holds (see Kataria and Khandakar (2021c), Eq. (2.5)):
\begin{equation}\label{asi1}
\operatorname{Cov}\left(Y_{\alpha}(s),Y_{\alpha}(t)\right)\sim \frac{1}{\Gamma^2(\alpha+1)}\left( \alpha s^{2\alpha}B(\alpha,\alpha+1)-\frac{\alpha^{2}s^{\alpha+1}}{(\alpha+1)t^{1-\alpha}}\right).
\end{equation}  
  
\subsection{The GCP and its fractional version}
Here, we give some known results for the GCP and its fractional version, the GFCP (see Di Crescenzo {\it et al.} (2016), Kataria and Khandakar (2021c)). 

The state probabilities $p(n,t)=\mathrm{Pr}\{M(t)=n\}$ of GCP are obtained by Di Crescenzo {\it et al.} (2016) which can be represented as follows:
\begin{equation}\label{p(n,t)}
p(n,t)=\sum_{\Omega(k,n)}\prod_{j=1}^{k}\frac{(\lambda_{j}t)^{x_{j}}}{x_{j}!}e^{-\Lambda t},\ \ n\ge 0,
\end{equation}
where $\Omega(k,n)\coloneqq\{(x_{1},x_{2},\dots,x_{k}):x_{1}+2x_{2}+\dots+kx_{k}=n,\ x_{j}\in\mathbb{N}_0\ \forall\ 1\leq j\leq k\}$.

Its pgf $G(u,t)=\mathbb{E}\left(u^{M(t)}\right)$ is given by
\begin{equation}\label{pgfmt}
G(u,t)=\exp\left(-\sum_{j=1}^{k}\lambda_{j}(1-u^{j})t\right),\  \ |u|\le 1
\end{equation}

and its characteristic function $\psi(\xi,t)=\mathbb{E}\left(e^{\omega\xi M(t)}\right)$, $\omega=\sqrt{-1}$  is given by
\begin{equation}\label{wsjh56}
\psi(\xi,t)=\exp\left\{{-t\left(\Lambda-\sum_{j=1}^{k}e^{\omega\xi j}\lambda_{j}\right)}\right\},\ \ \xi\in\mathbb{R}.
\end{equation}

The following limiting result holds for it:
\begin{equation}\label{limit}
\lim_{t\to\infty}\frac{M(t)}{t}=\sum_{j=1}^{k}j\lambda_{j},\ \ \text{in probability}.
\end{equation}
Its mean and variance are obtained by Di Crescenzo {\it et al.} (2016) in the following form:
\begin{equation*}
\mathbb{E}\left(M(t)\right)=r_{1}t\ \ \text{and}\ \  \operatorname{Var}\left(M(t)\right)=r_{2}t,
\end{equation*}
where
\begin{equation}\label{r1r2}
r_{1}=\sum_{j=1}^{k}j\lambda_{j} \ \  \text{and}\ \  r_{2}=\sum_{j=1}^{k}j^{2}\lambda_{j}.
\end{equation}

The mean and covariance of GFCP  are given by (see Kataria and Khandakar (2021c))
\begin{align}
\mathbb{E}\left(M^{\alpha}(t)\right)&=r_{1}\mathbb{E}\left(Y_{\alpha}(t)\right),\label{es}\\	
\operatorname{Cov}\left(M^{\alpha}(s),M^{\alpha}(t)\right)&=r_{2}\mathbb{E}\left(Y_{\alpha}(s)\right)+r_{1}^{2}\operatorname{Cov}\left(Y_{\alpha}(s),Y_{\alpha}(t)\right),\ \ 0<s\leq t.\label{cs}
\end{align}

\section{Generalized Skellam process and its fractional extension}
In this section, we introduce a Skellam type variant of the GCP, namely, the generalized Skellam process (GSP). We denote it by $\{\mathcal{S}(t)\}_{t\ge0}$ and define it as  
\begin{equation*}
\mathcal{S}(t)\coloneqq M_{1}(t)-M_{2}(t),
\end{equation*}
where  $\{M_{1}(t)\}_{t\ge0}$ and $\{M_{2}(t)\}_{t\ge0}$ are independent generalized counting processes with positive rates $\lambda_{j}$'s and $\mu_{j}$'s, $j=1,2,\dots,k$, respectively. 

For $k=1$, the GSP reduces to the Skellam process introduced and studied by Barndorff-Nielsen {\it et al.} (2012). For $\lambda_{j}=\lambda$ and $\mu_{j}=\mu$, $j=1,2,\dots,k$, the GSP reduces to the Skellam process of order $k$ (see Gupta {\it et al.} (2020)).

Using (\ref{pgfmt}), the pgf of GSP can be obtained as follows:
\begin{equation}\label{pgfgsp}
G_{\mathcal{S}}(u,t)=\exp\left\{-t\sum_{j=1}^{k}\left(\lambda_{j}(1-u^{j})+\mu_{j}(1-u^{-j})\right)\right\}.
\end{equation}
It satisfies the following differential equation:
\begin{equation*}
\frac{\mathrm{d}}{\mathrm{d}t}	G_{\mathcal{S}}(u,t)=\sum_{j=1}^{k}\left(\lambda_{j}(u^{j}-1)+\mu_{j}(u^{-j}-1)\right)G_{\mathcal{S}}(u,t)
, \ \ G_{\mathcal{S}}(u,0)=1.
\end{equation*}
\begin{remark}
On substituting	$\lambda_{j}=\lambda$ and $\mu_{j}=\mu$ for all $j=1,2,\dots,k$ in (\ref{pgfgsp}), we get the pgf of Skellam process of order $k$ (see Gupta {\it et al.} (2020), Eq. (39)).
\end{remark}
Similarly, on using (\ref{wsjh56}), the characteristic function of GSP can be obtained as
\begin{equation*}
\psi_{\mathcal{S}}(\xi,t)=	 \exp\left\{-t\left(\Lambda+\bar{\Lambda}-\sum_{j=1}^{k}e^{\omega\xi j}\lambda_{j}-\sum_{j=1}^{k}e^{-\omega \xi 
j}\mu_{j}\right)\right\},
\end{equation*}
where $\Lambda=\sum_{j=1}^{k}\lambda_{j}$ and $\bar{\Lambda}=\sum_{j=1}^{k}\mu_{j}$. Thus, the L\'evy measure of GSP is given by $\Pi_{\mathcal{S}}(\mathrm{d}x)=\sum_{j=1}^{k}\lambda_{j}\delta_{j}\mathrm{d}x+\sum_{j=1}^{k}\mu_{j}\delta_{-j}\mathrm{d}x$, where $\delta_{j}$'s are Dirac measures.

Let $m_{1}=\sum_{j=1}^{k}j(\lambda_{j}-\mu_{j})$ and $m_{2}=\sum_{j=1}^{k}j^{2}(\lambda_{j}+\mu_{j})$. The mean, variance and covariance of GSP are given by
\begin{equation*}
\mathbb{E}\left(\mathcal{S}(t)\right)=m_{1}t,\ \ 
\operatorname{Var}\left(\mathcal{S}(t)\right)=m_{2}t, \ \
\operatorname{Cov}\left(\mathcal{S}(s),\mathcal{S}(t)\right)=m_{2}\min\{s,t\}.
\end{equation*}
The GSP exhibits the overdispersion as $\operatorname{Var}\left(\mathcal{S}(t)\right)-\mathbb{E}\left(\mathcal{S}(t)\right)>0$ for all $t>0$.

The following definition of LRD and SRD property will be used (see  Maheshwari and Vellaisamy (2016)):
\begin{definition}
Let $s>0$ be fixed and $\{X(t)\}_{t\ge0}$ be a stochastic process such that
\begin{equation*}
\operatorname{Corr}(X(s),X(t))\sim c(s)t^{-\gamma},\ \ \text{as}\ t\to\infty,
\end{equation*}
for some $c(s)>0$. The process $\{X(t)\}_{t\ge0}$ has the LRD property if $\gamma\in(0,1)$ and SRD property if $\gamma\in (1,2)$.	.
\end{definition}
\begin{remark}
For fixed $s$ and large $t$, the correlation function of GSP has the following asymptotic behaviour:
\begin{equation*}
\operatorname{Corr}\left(\mathcal{S}(s),\mathcal{S}(t)\right)\sim\sqrt{s}t^{-1/2}.
\end{equation*}
Thus, it exhibits the LRD property.
\end{remark}
On using (\ref{limit}), we obtain the following limiting result for GSP:
\begin{equation}\label{limitS}
\lim\limits_{t\to\infty}\frac{\mathcal{S}(t)}{t}=\sum_{j=1}^{k}j(\lambda_{j}-\mu_{j}),\ \ \text{in probability}.
\end{equation} 

\begin{theorem}
For any $n\in\mathbb{Z}$, the state probability $q(n,t)=\mathrm{Pr}\{\mathcal{S}(t)=n\}$ of GSP is given by
\begin{equation}\label{qkl11}
q(n,t)=e^{-(\Lambda+\bar{\Lambda})t}\left(\Lambda/\bar{\Lambda}\right)^{n/2}I_{|n|}\left(2t\sqrt{\Lambda \bar{\Lambda}}\right).
\end{equation}
\end{theorem}
\begin{proof}
As $\{M_{1}(t)\}_{t\ge0}$ and $\{M_{2}(t)\}_{t\ge0}$ are independent, we have
\begin{align*}
q(n,t)&=\sum_{m=0}^{\infty}\mathrm{Pr}\{M_{1}(t)=m+n\}\mathrm{Pr}\{M_{2}(t)=m\}\mathbb{I}_{\{n\ge0\}}\\
&\ \ +\sum_{m=0}^{\infty}\mathrm{Pr}\{M_{2}(t)=m+|n|\}\mathrm{Pr}\{M_{1}(t)=m\}\mathbb{I}_{\{n<0\}}\\
&=\sum_{m=0}^{\infty}\left(\sum_{ \Omega(k,m+n)}\prod_{j=1}^{k}\frac{(\lambda_{j}t)^{x_{j}}}{x_{j}!}e^{-\Lambda t}\right)\left(\sum_{ \Omega(k,m)}\prod_{j=1}^{k}\frac{(\mu_{j}t)^{x_{j}}}{x_{j}!}e^{-\bar{\Lambda} t}\right)\mathbb{I}_{\{n\ge0\}}\\
&\ \ +\sum_{m=0}^{\infty}\left(\sum_{ \Omega(k,m+|n|)}\prod_{j=1}^{k}\frac{(\mu_{j}t)^{x_{j}}}{x_{j}!}e^{-\bar{\Lambda} t}\right)\left(\sum_{ \Omega(k,m)}\prod_{j=1}^{k}\frac{(\lambda_{j}t)^{x_{j}}}{x_{j}!}e^{-\Lambda t}\right)\mathbb{I}_{\{n<0\}}
\end{align*}
where we have used (\ref{p(n,t)}) in the last step. Here, $\mathbb{I}_{A}$ denotes the indicator function for the set $A$.

Let $n\geq0$. On setting $x_{j}=m_{j}$ and $m=x+\sum_{j=1}^{k}(j-1)m_{j}$, we get 
\begin{align}\label{key2wsw1}
q(n,t)&=e^{-(\Lambda+\bar{\Lambda})t}\sum_{x=0}^{\infty}\frac{t^{n+2x}}{(n+x)!x!}\left(\sum_{\sum_{i=1}^km_{i}=n+x}(n+x)!\prod_{j=1}^{k}\frac{\lambda_{j}^{m_{j}}}{m_{j}!}\right) \left(\sum_{\sum_{i=1}^km_{i}=x}x!\prod_{j=1}^{k}\frac{\mu_{j}^{m_{j}}}{m_{j}!}\right)\nonumber\\
&=e^{-(\Lambda+\bar{\Lambda})t}\sum_{x=0}^{\infty}\frac{t^{n+2x}}{(n+x)!x!}\Lambda^{n+x}\bar{\Lambda}^{x}\nonumber\\
&=e^{-(\Lambda+\bar{\Lambda})t}\left(\Lambda/\bar{\Lambda}\right)^{n/2}I_{n}\left(2t\sqrt{\Lambda \bar{\Lambda}}\right),
\end{align}
where we have used multinomial theorem in the penultimate step and the definition (\ref{bessel}) of modified Bessel function of first kind in the last step.

 Similarly, for $n<0$, we get
\begin{equation}\label{key2wsw2}
q(n,t)=e^{-(\Lambda+\bar{\Lambda})t}\left(\Lambda/\bar{\Lambda}\right)^{n/2}I_{-n}\left(2t\sqrt{\Lambda \bar{\Lambda}}\right).
\end{equation}
On combining (\ref{key2wsw1}) and (\ref{key2wsw2}), we get the required result.	
\end{proof}

On differentiating (\ref{qkl11}) and then using (\ref{besseli}), we obtain the following result:
\begin{corollary}
The state probabilities of GSP satisfy the following system of differential equations:
\begin{equation}\label{pnt}
\frac{\mathrm{d}}{\mathrm{d}t}q(n,t)=\Lambda\left(q(n-1,t)-q(n,t)\right)-\bar{\Lambda}\left(q(n,t)-q(n+1,t)\right),\ \ n\in \mathbb{Z},
\end{equation}
with initial conditions $q(0,0)=1$ and $q(n,0)=0$, $n\neq 0$.
\end{corollary}
\begin{remark}
On substituting $k=1$ in (\ref{pnt}), we get the governing system of differential equations for the pmf of Skellam process (see Kerss {\it et al.} (2014), Eq. (2.4)).
\end{remark}
\subsection{Generalized fractional Skellam process}
Here, we consider a fractional version of the GSP. We call it the generalized fractional Skellam process (GFSP) and denote it by $\{\mathcal{S}^{\alpha}(t)\}_{t\ge0}$, $0<\alpha\leq 1$. It is defined as 
\begin{equation}\label{def}
\mathcal{S}^{\alpha}(t)\coloneqq
\begin{cases*}
\mathcal{S}(Y_{\alpha}(t)),\ \ 0<\alpha<1,\\
\mathcal{S}(t),\ \ \alpha=1,
\end{cases*}
\end{equation}
where the GSP $\{\mathcal{S}(t)\}_{t\ge0}$ is independent of the inverse stable subordinator $\{Y_{\alpha}(t)\}_{t\ge0}$. It is important to note that the GFSP is a Skellam type version of the GFCP.

For $k=1$, the GFSP reduces to the fractional Skellam process of type II (see Kerss {\it et al.} (2014)). For $\lambda_{j}=\lambda$ and $\mu_{j}=\mu$, $j=1,2,\dots,k$, the GFSP reduces to the fractional Skellam process of order $k$, introduced and studied by Kataria and Khandakar (2021b).

In the following result, we obtain the governing system of fractional differential equations for the state probabilities of GFSP.
\begin{proposition}
The state probabilities $q^{\alpha}(n,t)=\mathrm{Pr}\{\mathcal{S}^{\alpha}(t)=n\}$, $n\in \mathbb{Z}$ of  GFSP solves the following system of fractional differential equations: 
\begin{equation}\label{pmf}
\partial_{t}^{\alpha}q^{\alpha}(n,t)=\Lambda\left(q^{\alpha}(n-1,t)-q^{\alpha}(n,t)\right)-\bar{\Lambda}\left(q^{\alpha}(n,t)-q^{\alpha}(n+1,t)\right),
\end{equation}
with initial conditions $q^{\alpha}(0,0)=1$ and $q^{\alpha}(n,0)=0, \ n\neq0$. 
\end{proposition}
\begin{proof}
From (\ref{def}), we have
\begin{equation}\label{ci}
q^{\alpha}(n,t)=\int_{0}^{\infty}q(n,u)h_{\alpha}(u,t)\,\mathrm{d}u,
\end{equation}	
where $q(n,.)$ is the pmf of $\{\mathcal{S}(t)\}_{t\ge0}$ and $h_{\alpha}(.,t)$ is the probability density function (pdf) of $\{Y_{\alpha}(t)\}_{t\ge0}$. Note that $q^{\alpha}(n,0)=q(n,0)$ as $h_{\alpha}(u,0)=\delta_0(u)$. On taking the R-L fractional derivative in (\ref{ci}), we get
\begin{align}\label{Rl}
D_{t}^{\alpha}q^{\alpha}(n,t)&=-\int_{0}^{\infty}q(n,u)\frac{\partial}{\partial u}h_{\alpha}(u,t)\mathrm{d}u\nonumber\\
&=q(n,0)h_{\alpha}(0+,t)+\int_{0}^{\infty}h_{\alpha}(u,t)\frac{\mathrm{d}}{\mathrm{d} u}q(n,u)\,\mathrm{d}u \nonumber\\
&=q(n,0)\frac{t^{-\alpha}}{\Gamma(1-\alpha)}+\int_{0}^{\infty}h_{\alpha}(u,t)\frac{\mathrm{d}}{\mathrm{d} u}q(n,u)\,\mathrm{d}u,
\end{align}
where we have used the following results: $D_{t}^{\alpha}h_{\alpha}(u,t)=-\frac{\partial}{\partial u}h_{\alpha}(u,t)$ and $h_{\alpha}(0+,t)=t^{-\alpha}/\Gamma(1-\alpha)$ (see Meerschaert and Straka  (2013)). 

On using (\ref{relation}) and (\ref{pnt}) in (\ref{Rl}), we get
\begin{align*}
\partial_{t}^{\alpha}q^
{\alpha}(n,t)&=\int_{0}^{\infty}h_{\alpha}(u,t)\frac{\mathrm{d}}{\mathrm{d} u}q(n,u)\,\mathrm{d}u\\
&=\int_{0}^{\infty}\left(\Lambda\left(q(n-1,u)-q(n,u)\right)-\bar{\Lambda}\left(q(n,u)-q(n+1,u)\right)\right)h_{\alpha}(u,t)\,\mathrm{d}u,
\end{align*}
which reduces to the required result on using (\ref{ci}).
\end{proof}	
\begin{remark}
For $k=1$, we get $\Lambda=\lambda_{1}$ and $\bar{\Lambda}=\mu_{1}$.  So, in this case the system given in (\ref{pmf}) reduces to the system of governing fractional differential equations for the state probabilities of fractional Skellam process of type II (see Kerss {\it et al.} (2014), Eq. (3.5)).
\end{remark}
\begin{remark}
The pdf of inverse stable subordinator can be expressed in terms of the Wright function as follows (see Meerschaert {\it et al.} (2015), Eq. (4.7)):
\begin{equation}\label{wright1s}
h_{\alpha}(u,t)=t^{-\alpha}M_{\alpha}\left(ut^{-\alpha}\right).
\end{equation}
On substituting (\ref{qkl11}) and (\ref{wright1s}) in (\ref{ci}), we get an integral representation of the state probabilities of GFSP as follows:
\begin{equation*}
q^{\alpha}(n,t)=
t^{-\alpha}\left(\Lambda/\bar{\Lambda}\right)^{n/2}\int_{0}^{\infty}e^{-(\Lambda+\bar{\Lambda})u}I_{|n|}\left(2u\sqrt{\Lambda \bar{\Lambda}}\right)M_{\alpha}(ut^{-\alpha})\mathrm{d}u,\ \ n\in\mathbb{Z}.
\end{equation*}
\end{remark}	
\begin{proposition}
The pgf of GFSP  is given by
\begin{equation}\label{pgf}
G^{\alpha}_{\mathcal{S}}(u,t)=	E_{\alpha,1}\left(\sum_{j=1}^{k}\left(\lambda_{j}(u^{j}-1)+\mu_{j}(u^{-j}-1)\right)t^{\alpha}\right).
\end{equation}
\end{proposition}
\begin{proof}
Using (\ref{pgfgsp}), we get
\begin{align*}
G^{\alpha}_{\mathcal{S}}(u,t)&=\int_{0}^{\infty}G_{\mathcal{S}}(u,x)h_{\alpha}(x,t)\mathrm{d}x\\
&=\int_{0}^{\infty}\exp\left\{-x\sum_{j=1}^{k}\left(\lambda_{j}(1-u^{j})+\mu_{j}(1-u^{-j})\right)\right\}h_{\alpha}(x,t)\mathrm{d}x\\
&=E_{\alpha,1}\left(\sum_{j=1}^{k}\left(\lambda_{j}(u^{j}-1)+\mu_{j}(u^{-j}-1)\right)t^{\alpha}\right).
\end{align*}
This completes the proof.
\end{proof}
The pgf of GFSP solves the following fractional differential equation:
\begin{equation*}
\partial_{t}^{\alpha}G^{\alpha}_{\mathcal{S}}(u,t)=\sum_{j=1}^{k}\left(\lambda_{j}(u^{j}-1)+\mu_{j}(u^{-j}-1)\right)G^{\alpha}_{\mathcal{S}}(u,t),\ \ G^{\alpha}_{\mathcal{S}}(u,0)=1
\end{equation*}
which is due to the fact that the Mittag-Leffler function is an eigenfunction of the Caputo fractional derivative.
\begin{remark}
On substituting	$\lambda_{j}=\lambda$ and $\mu_{j}=\mu$ for all $j=1,2,\dots,k$ in (\ref{pgf}), we get the pgf of fractional Skellam process of order $k$ (see Kataria and Khandakar (2021b), Proposition 2.2).
\end{remark}
Next we obtain the factorial moments of GFSP by using its pgf.
\begin{proposition}\label{p3.3}
The $r$th factorial moment of GFSP, that is, $\Psi^\alpha(r,t)=	\mathbb{E}(\mathcal{S}^{\alpha}(t)(\mathcal{S}^{\alpha}(t)-1)\cdots(\mathcal{S}^{\alpha}(t)-r+1))$, $r\ge1$, is given by
\begin{equation*}
\Psi^\alpha(r,t)=r!\sum_{n=1}^{r}\frac{t^{n\alpha}}{\Gamma(n\alpha+1)}\underset{m_i\in\mathbb{N}}{\underset{\sum_{i=1}^nm_i=r}{\sum}}\prod_{\ell=1}^{n}\left(\frac{1}{m_\ell!}\sum_{j=1}^{k}\left((j)_{m_\ell}\lambda_{j}+(-1)^{m_\ell}j^{(m_\ell)}\mu_{j}\right)\right),
\end{equation*}
where $(j)_{m_\ell}=j(j-1)\cdots(j-m_\ell+1)$ denotes the falling factorial and $j^{(m_\ell)}=j(j+1)\cdots(j+m_\ell-1)$ denotes the rising factorial.
\end{proposition}
\begin{proof}
Let $\zeta(u)=\sum_{j=1}^{k}\left(\lambda_{j}(u^{j}-1)+\mu_{j}(u^{-j}-1)\right)$. On using the $r$th derivative of composition of two functions (see Johnson (2002), Eq. (3.3))	in (\ref{pgf}), we get
\begin{align}\label{mkgtrr4543t}
\Psi^\alpha(r,t)&=\frac{\partial^{r}G^{\alpha}_{\mathcal{S}}(u,t)}{\partial u^{r}}\bigg|_{u=1}\nonumber\\
&=\sum_{n=0}^{r}E^{n+1}_{\alpha,n\alpha+1}(t^{\alpha}\zeta(u))\sum_{m=0}^{n}\frac{n!}{m!(n-m)!}\left(-t^{\alpha}\zeta(u)\right)^{n-m}\frac{\mathrm{d}^{r}}{\mathrm{d}u^{^{r}}}\left(t^{\alpha}\zeta(u)\right)^{m}\bigg|_{u=1}\nonumber\\
&=\sum_{n=0}^{r}\frac{t^{n\alpha }}{\Gamma(n\alpha+1)}\frac{\mathrm{d}^{r}}{\mathrm{d}u^{^{r}}}\left(\zeta(u)\right)^{n}\bigg|_{u=1},
\end{align}
where we have used (\ref{2.2desww}) in the penultimate step. On using the following result (see Johnson (2002), Eq. (3.6))
\begin{equation*}
\frac{\mathrm{d}^{r}}{\mathrm{d}w^{^{r}}}(g(w))^{n}=\underset{m_i\in\mathbb{N}_0}{\underset{m_{1}+m_{2}+\dots+m_{n}=r}{\sum}}\frac{r!}{m_1!m_2!\cdots m_n!}g^{(m_{1})}(w)g^{(m_{2})}(w)\cdots g^{(m_{n})}(w),
\end{equation*}
we have	
\begin{align}\label{ccc}
\frac{\mathrm{d}^{r}}{\mathrm{d}u^{^{r}}}\left(\zeta(u)\right)^{n}\bigg|_{u=1}&=r!\underset{m_i\in\mathbb{N}_0}{\underset{\sum_{i=1}^nm_i=r}{\sum}}\prod_{\ell=1}^{n}\frac{1}{m_\ell!}\frac{\mathrm{d}^{m_{\ell}}}{\mathrm{d}u^{{m_{\ell}}}}\zeta(u)\bigg|_{u=1}\nonumber\\
&=r!\underset{m_i\in\mathbb{N}}{\underset{\sum_{i=1}^nm_i=r}{\sum}}\prod_{\ell=1}^{n}\left(\frac{1}{m_\ell!}\sum_{j=1}^{k}\left((j)_{m_\ell}\lambda_{j}+(-1)^{m_\ell}j^{(m_\ell)}\mu_{j}\right)\right).
\end{align}	
As the right hand side of (\ref{ccc}) vanishes for $n=0$, the proof follows by substituting (\ref{ccc}) in (\ref{mkgtrr4543t}).	
\end{proof}
The mean, variance and  covariance of GFSP are obtained by using Theorem 2.1 of Leonenko {\it et al.} (2014) as follows: Let $0<s\leq t$. Then,
\begin{align}
\mathbb{E}\left(\mathcal{S}^{\alpha}(t)\right)&=m_{1}\mathbb{E}\left(Y_{\alpha}(t)\right),\nonumber
\\
\operatorname{Var}\left(\mathcal{S}^{\alpha}(t)\right)&=m_{2}\mathbb{E}\left(Y_{\alpha}(t)\right)+m_{1}^{2}\operatorname{Var}\left(Y_{\alpha}(t)\right),\label{vs2}\\	
\operatorname{Cov}\left(\mathcal{S}^{\alpha}(s),\mathcal{S}^{\alpha}(t)\right)&=m_{2}\mathbb{E}\left(Y_{\alpha}(s)\right)+m_{1}^{2}\operatorname{Cov}\left(Y_{\alpha}(s),Y_{\alpha}(t)\right).\label{cs3}
\end{align}
The GFSP exhibits overdispersion as $\operatorname{Var}\left(\mathcal{S}^{\alpha}(t)\right)-\mathbb{E}\left(\mathcal{S}^{\alpha}(t)\right)>0$ for all $t>0$.
\begin{proposition}
The GFSP exhibits the LRD property.
\end{proposition}
\begin{proof}
From (\ref{vs2}) and (\ref{cs3}), we get
\begin{equation*}
\operatorname{Corr}	\left(\mathcal{S}^{\alpha}(s),\mathcal{S}^{\alpha}(t)\right)=\frac{m_{2}\mathbb{E}\left(Y_{\alpha}(s)\right)+m_{1}^{2}\operatorname{Cov}\left(Y_{\alpha}(s),Y_{\alpha}(t)\right)}{\sqrt{\operatorname{Var}\left(\mathcal{S}^{\alpha}(s)\right)}\sqrt{m_{2}\mathbb{E}\left(Y_{\alpha}(t)\right)+m_{1}^{2}\operatorname{Var}\left(Y_{\alpha}(t)\right)}}.
\end{equation*}
On using (\ref{meani}), (\ref{xswe331}) and (\ref{asi1}) for fixed $s$ and large $t$, we get
\begin{align*}
\operatorname{Corr}	\left(\mathcal{S}^{\alpha}(s),\mathcal{S}^{\alpha}(t)\right)&\sim \frac{m_{2}\Gamma^2(\alpha+1)\mathbb{E}\left(Y_{\alpha}(s)\right)+m_{1}^{2}\left( \alpha s^{2\alpha}B(\alpha,\alpha+1)-\frac{\alpha^{2}s^{\alpha+1}}{(\alpha+1)t^{1-\alpha}}\right)}{\Gamma^2(\alpha+1)\sqrt{\operatorname{Var}\left(\mathcal{S}^{\alpha}(s)\right)}\sqrt{\frac{m_{2}t^{\alpha}}{\Gamma(\alpha+1)}+\frac{2m_{1}^{2}t^{2\alpha}}{\Gamma(2\alpha+1)}-\frac{m_{1}^{2}t^{2\alpha}}{\Gamma^{2}(\alpha+1)}}}\\
&\sim c_{0}(s)t^{-\alpha},
\end{align*}
where
\begin{equation*}
c_{0}(s)=	\frac{m_{2}\Gamma^2(\alpha+1)\mathbb{E}\left(Y_{\alpha}(s)\right)+m_{1}^{2}\alpha s^{2\alpha}B(\alpha,\alpha+1)}{\Gamma^2(\alpha+1)\sqrt{\operatorname{Var}\left(\mathcal{S}^{\alpha}(s)\right)}\sqrt{\frac{2m_{1}^{2}}{\Gamma(2\alpha+1)}-\frac{m_{1}^{2}}{\Gamma^{2}(\alpha+1)}}}.
\end{equation*}
As $0<\alpha<1$, it follows that the GFSP has the LRD property.
\end{proof}
\begin{remark}
For a fixed $h>0$, the increment of GFSP is defined as
\begin{equation*}
Z_{h}^{\alpha}(t)\coloneqq\mathcal{S}^{\alpha}(t+h)-\mathcal{S}^{\alpha}(t).
\end{equation*}
It can be shown that the increment process $\{Z_{h}^{\alpha}(t)\}_{t\ge0}$ exhibits the SRD property. The proof follows similar lines to that of Theorem 1 of Maheshwari and Vellaisamy (2016).
\end{remark}
\begin{proposition}
The one-dimensional distributions of GFSP are not infinitely divisible.
\end{proposition}
\begin{proof}
Using the self-similarity property of	$\{Y_{\alpha}(t)\}_{t\ge0}$,  we get $\mathcal{S}^{\alpha}(t)\overset{d}{=}\mathcal{S}\left(t^{\alpha}Y_{\alpha}(1)\right)$. Thus,
\begin{align*}
\lim_{t\to\infty}\frac{\mathcal{S}^{\alpha}(t)}{t^{\alpha}}&\overset{d}{=}\lim_{t\to\infty}\frac{\mathcal{S}\left(t^{\alpha}Y_{\alpha}(1)\right)}{t^{\alpha}}\\
&=Y_{\alpha}(1)\lim_{t\to\infty}\frac{\mathcal{S}\left(t^{\alpha}Y_{\alpha}(1)\right)}{t^{\alpha}Y_{\alpha}(1)}\\
&\overset{d}{=}Y_{\alpha}(1)\sum_{j=1}^{k}j(\lambda_{j}-\mu_{j}),
\end{align*}
where we have used $(\ref{limitS})$.

Suppose on contrary $\mathcal{S}^{\alpha}(t)$ is infinitely divisible. Then, by using Proposition 2.1 of Steutel and van Harn (2004), it follows that $\mathcal{S}^{\alpha}(t)/t^{\alpha}$ is  infinitely divisible. It is known that the limit of a sequence of infinitely divisible random variables is infinitely divisible (see Steutel and van Harn (2004), Proposition 2.2). This implies that  $Y_{\alpha}(1)$ is infinitely divisible. This leads to a contradiction as $Y_{\alpha}(1)$ is not infinitely divisible (see Vellaisamy and Kumar (2018)).
\end{proof}

\section{GFCP time-changed by a L\'evy subordinator}
In this section, we consider a time-changed version of the GFCP. We call it the time-changed  generalized fractional counting process-I (TCGFCP-I) and denote it by $\{\mathcal{Z}_{f}^{\alpha}(t)\}_{t\ge0}$, $0<\alpha\le1$.

Let $\{D_f (t)\}_{t\ge0}$ be a L\'evy subordinator such that $\mathbb{E}\left(D^r_f(t)\right)<\infty$ for all $r>0$. We define the TCGFCP-I as 
\begin{equation}\label{qws11ww1}
\mathcal{Z}_{f}^{\alpha}(t)\coloneqq M^{\alpha}(D_f (t)),
\end{equation}
where the GFCP $\{M^{\alpha}(t)\}_{t\ge0}$ is independent of $\{D_f (t)\}_{t\ge0}$. 

For $\alpha=1$, the TCGFCP-I reduces to  a time-changed version of the GCP, namely, the time-changed  generalized  counting process-I (TCGCP-I) $\{\mathcal{Z}_{f}(t)\}_{t\ge0}$, that is, 
\begin{equation}\label{zft}
\mathcal{Z}_{f}(t)\coloneqq M(D_f (t)).
\end{equation}
For $k=1$, the TCGFCP-I reduces to TCFPP-I, a time-changed version of the TFPP which is introduced and studied by Maheshwari and Vellaisamy (2019). Also, for $k=1$, the TCGCP-I reduces to a time-changed version of the Poisson process which is introduced and studied by Orsingher and Toaldo (2015) with the condition that the involved subordinator has finite moments of all order. 
 
On taking $\lambda_{j}=\lambda$ for all $j=1,2,\dots,k$, the GFCP and GCP reduces to the time fractional Poisson process of order $k$ (TFPPoK) and Poisson process of order $k$ (PPoK) (see Kataria and Khandakar (2021c)), respectively. For such choice of $\lambda_{j}$'s, the TCGFCP-I and TCGCP-I reduces to a time-changed version of the TFPPoK and PPoK (see Sengar {\it et al.} (2020)), respectively. For $\lambda_{j}=\lambda(1-\rho)\rho^{j-1}/(1-\rho^{k})$, $0\leq\rho<1$, $j=1,2,\dots,k$, the GFCP and GCP reduces to the fractional P\'olya-Aeppli process of order $k$ (FPAPoK) and P\'olya-Aeppli process of order $k$ (PAPoK) (see Kataria and Khandakar (2021c)), respectively. In this case, the TCGFCP-I and TCGCP-I reduces to a time-changed version of the FPAPoK and PAPoK, respectively. 

For $\lambda_{j}=\lambda(1-\rho)\rho^{j-1}$, $0\leq\rho<1$, for all $j\ge1$ with $k\to \infty$, the GFCP and GCP reduces to the fractional P\'olya-Aeppli process (FPAP) and P\'olya-Aeppli process (PAP) (see Kataria and Khandakar (2021c)), respectively. Thus, the TCGFCP-I and TCGCP-I reduces to a time-changed version of the FPAP and PAP, respectively. Also, when $\lambda_{j}=\beta_{j-1}-\beta_{j}$, $j\ge 1$ where the sequence $\{\beta_{j}\}_{j\in\mathbb{Z}}$ is such that $\beta_{j}=0$ for all $j<0$ and $\beta_{j}>\beta_{j+1}>0$ for all $j\geq0$ with $\lim\limits_{j\to\infty}\beta_{j+1}/\beta_{j}<1$, the GFCP and GCP reduces to the convoluted fractional Poisson process and convoluted Poisson process (see Kataria and Khandakar (2021a)), respectively, as $k\to \infty$. Thus, in this case, the TCGFCP-I and TCGCP-I reduces to a time-changed version of the convoluted fractional Poisson process and convoluted Poisson process, respectively.

\begin{theorem}\label{thm1}
The pmf $p_{f}(n,t)=\mathrm{Pr}\{\mathcal{Z}_{f}(t)=n\}$ of TCGCP-I is given by
\begin{equation}\label{fd12dfqq}
p_{f}(n,t)=\sum_{\Omega(k,n)}\prod_{j=1}^{k}\frac{\lambda_{j}^{x_{j}}}{x_{j}!}\mathbb{E}\left(e^{-\Lambda D_{f}(t)}D_{f}^{s_{k}}(t)\right),\ \ n\geq0,
\end{equation}
where $s_{k}=x_{1}+x_{2}+\dots+x_{k}$ and $\Omega(k,n)$ is as given in (\ref{p(n,t)}).
\end{theorem}
\begin{proof}
Let $h_{f}(x,t)$ be the pdf of $D_f (t)$, and recall that $p(n,x)$ denotes the pmf of GCP. From (\ref{zft}), we get
\begin{align}\label{sqaa7}
p_{f}(n,t)&=\int_0^\infty p(n,x)h_{f}(x,t)\,\mathrm{d}x\nonumber\\
&=\int_0^\infty\sum_{\Omega(k,n)}\prod_{j=1}^{k}\frac{(\lambda_{j}x)^{x_{j}}}{x_{j}!}e^{-\Lambda x}h_{f}(x,t)\,\mathrm{d}x, \ \ \  (\text{using (\ref{p(n,t)})})\\
&=\sum_{\Omega(k,n)}\prod_{j=1}^{k}\frac{\lambda_{j}^{x_{j}}}{x_{j}!}\mathbb{E}\left(e^{-\Lambda D_{f}(t)}D_{f}^{s_{k}}(t)\right).\nonumber
\end{align}
This completes the proof.
\end{proof}
Note that
\begin{align*}
\sum_{n=0}^{\infty}p_{f}(n,t)&=\sum_{n=0}^{\infty}\sum_{s_k=0}^n\mathbb{E}\left(e^{-\Lambda D_{f}(t)}D_{f}^{s_{k}}(t)\right)\underset{x_{1}+2x_{2}+\dots+kx_{k}=n}{\sum_{ x_{1}+x_{2}+\dots+x_{k}=s_k}}\prod_{j=1}^{k}\frac{\lambda_{j}^{x_{j}}}{x_{j}!}\\
&=\sum_{s_k=0}^{\infty}\mathbb{E}\left(e^{-\Lambda D_{f}(t)}D_{f}^{s_{k}}(t)\right)\sum_{n=s_k}^\infty\underset{x_{1}+2x_{2}+\dots+kx_{k}=n}{\sum_{ x_{1}+x_{2}+\dots+x_{k}=s_k}}\prod_{j=1}^{k}\frac{\lambda_{j}^{x_{j}}}{x_{j}!}\\
&=\sum_{s_k=0}^{\infty}\mathbb{E}\left(e^{-\Lambda D_{f}(t)}D_{f}^{s_{k}}(t)\right)\sum_{ x_{1}+x_{2}+\dots+x_{k}=s_k}\prod_{j=1}^{k}\frac{\lambda_{j}^{x_{j}}}{x_{j}!}\\
&=\sum_{s_k=0}^{\infty}\frac{\Lambda^{s_k}}{s_k!}\mathbb{E}\left(e^{-\Lambda D_{f}(t)}D_{f}^{s_k}(t)\right),\ \ \text{(using\ multinomial\ theorem)}\\
&=\int_{0}^{\infty}h_{f}(x,t)e^{-\Lambda x}\sum_{s_k=0}^{\infty}\frac{(\Lambda x)^{s_k}}{s_k!}\mathrm{d}x=\int_{0}^{\infty}h_{f}(x,t)\mathrm{d}x=1.
\end{align*}
Thus, $p_{f}(n,t)$ is indeed a pmf.
\begin{remark}
On substituting $\lambda_{j}=\lambda$, $j=1,2,\dots,k$ in (\ref{fd12dfqq}), we get
\begin{equation*}
p_{f}(n,t)\Big|_{\lambda_{j}=\lambda}=\sum_{ \Omega(k,n)}\frac{\lambda^{s_{k}}}{x_{1}!x_{2}!\dots x_{k}!}\mathbb{E}\left(e^{-k\lambda D_{f}(t)}D_{f}^{s_{k}}(t)\right),\ \ n\geq0,
\end{equation*}
which agrees with the pmf of a time-changed PPoK (see Sengar {\it et al.} (2020), Eq. (7)).
\end{remark}
\begin{remark}
From (\ref{sqaa7}), the pmf of TCGCP-I can alternatively be expressed as
\begin{align*}
p_{f}(n,t)&=\sum_{ \Omega(k,n)}\prod_{j=1}^{k}\frac{\lambda_{j}^{x_{j}}}{x_{j}!}\frac{(-1)^{s_{k}}}{\Lambda^{s_{k}}}\frac{\mathrm{d}^{s_{k}}}{\mathrm{d}v^{s_{k}}}\int_0^\infty e^{-\Lambda xv}h_{f}(x,t)\,\mathrm{d}x\Big|_{v=1}\\
&=\sum_{ \Omega(k,n)}\prod_{j=1}^{k}\frac{\lambda_{j}^{x_{j}}}{x_{j}!}\frac{(-1)^{s_{k}}}{\Lambda^{s_{k}}}\frac{\mathrm{d}^{s_{k}}}{\mathrm{d}v^{s_{k}}}e^{-tf(\Lambda v)}\Big|_{v=1}.
\end{align*}
For $k=1$, the above expression reduces to the pmf of a time-changed Poisson process (see Orsingher and Toaldo (2015), Eq. (2.4)).
\end{remark}

Next, we obtain the pgf $G_{f}(u,t)=\mathbb{E}\left(u^{\mathcal{Z}_{f}(t)}\right)$ of TCGCP-I.
\begin{proposition}
The pgf of TCGCP-I is given by
\begin{equation}\label{pgff}
G_{f}(u,t)=\exp\left\{-tf\left(\sum_{j=1}^{k}\lambda_{j}(1-u^{j})\right)\right\},\ \ |u|\le1.
\end{equation}
\end{proposition}
\begin{proof}
Using (\ref{pgfmt}), we get
\begin{align*}
G_{f}(u,t)&=\int_{0}^{\infty}G(u,x)h_{f}(x,t)\mathrm{d}x\\
&=\int_{0}^{\infty}\exp\left(-\sum_{j=1}^{k}\lambda_{j}(1-u^{j})x\right)h_{f}(x,t)\mathrm{d}x\\
&=\exp\left\{-tf\left(\sum_{j=1}^{k}\lambda_{j}(1-u^{j})\right)\right\}.
\end{align*}
This completes the proof.
\end{proof}
The pgf of TCGCP-I satisfies the following differential equation:
\begin{equation*}
\frac{\mathrm{d}}{\mathrm{d}t}G_{f}(u,t)=-f\left(\sum_{j=1}^{k}\lambda_{j}(1-u^{j})\right)G_{f}(u,t),\ \ G_{f}(u,0)=1.
\end{equation*}
\begin{remark}
On putting $k=1$ in (\ref{pgff}), we get the pgf of a time-changed Poisson process (see Orsingher and Toaldo (2015), Eq. (2.2)).
\end{remark}
\begin{remark}
We note that the TCGCP-I is equal in distribution to $\{\mathcal{X}_{f}(t)\}_{t\ge0}$ where $\mathcal{X}_{f}(t)=\sum_{j=1}^{k}jN_{j}(D_{f}(t))$,  a time-changed process introduced by Zuo {\it et al.} (2021). Here, for each $1\leq j\leq k$, $\{N_{j}(t)\}_{t\ge0}$ is a Poisson process with intensity $\lambda_{j}$ which is independent of the L\'evy subordinator
 $\{D_{f}(t)\}_{t\ge0}$. This holds due to the fact that $M(t)\stackrel{d}{=}\sum_{j=1}^{k}jN_{j}(t)$ (see Kataria and Khandakar (2021c)). 

The distribution of jumps of the process $\{\mathcal{X}_{f}(t)\}_{t\ge0}$ is given by Zuo {\it et al.} (2021), Eq. (4.2) as follows:
\begin{equation}\label{bgd43}
\mathrm{Pr}\{\mathcal{X}_{f}(h)=n\}=\begin{cases*}
1-hf(\Lambda)+o(h),\ \ n=0,\\
\displaystyle-h\sum_{ \Omega(k,n)}f^{(s_{k})}(\Lambda)\prod_{j=1}^{k}\frac{(-\lambda_{j})^{x_{j}}}{x_{j}!}+o(h),\ \ n\ge1.
\end{cases*}
\end{equation}
\end{remark}
Next, we obtain a version of the law of iterated logarithm for TCGFCP-I. 
\begin{theorem}
Let $f$ be a Bern\v stein function  associated with L\'evy subordinator $\{D_{f}(t)\}_{t\ge0}$ such that $\lim_{x\rightarrow 0+}f(\lambda x)/f(x)=\lambda^\theta$, $\lambda>0$ which implies that it is regularly varying at 0+ with index $0<\theta<1$. Also, let
\begin{equation*}
g(t)=\frac{\log\log t}{\phi(t^{-1}\log\log t)},\ \ t>e,
\end{equation*}
where $\phi$ is the inverse of $f$. Then, 
\begin{equation}\label{lil}
\liminf_{t\rightarrow\infty}\frac{\mathcal{Z}_{f}^{\alpha}(t)}{(g(t))^{\alpha}}\stackrel{d}{=}\sum_{j=1}^{k}j\lambda_{j} Y_{\alpha}(1)\theta^\alpha\left(1-\theta\right)^{\alpha(1-\theta)/\theta}.
\end{equation}	
\end{theorem}
\begin{proof}
From (\ref{ed}) and  (\ref{qws11ww1}), we get
\begin{equation*}
\mathcal{Z}_{f}^{\alpha}(t)\stackrel{d}{=}M(Y_{\alpha}(D_{f}(t)))\stackrel{d}{=}M(D^{\alpha}_{f}(t)Y_{\alpha}(1)),
\end{equation*}  
where we have used the self-similarity property of	$\{Y_{\alpha}(t)\}_{t\ge0}$. Thus,
\begin{align*}
\liminf_{t\rightarrow\infty}\frac{\mathcal{Z}_{f}^{\alpha}(t)}{(g(t))^{\alpha}}&\stackrel{d}{=}\liminf_{t\rightarrow\infty}\frac{M(D^{\alpha}_{f}(t)Y_{\alpha}(1))}{(g(t))^{\alpha}}\\
&=\liminf_{t\rightarrow\infty}\left(\frac{M(D^{\alpha}_{f}(t)Y_{\alpha}(1))}{D^{\alpha}_{f}(t)Y_{\alpha}(1)}\right)\frac{D^{\alpha}_{f}(t)Y_{\alpha}(1)}{(g(t))^{\alpha}}\\
&\stackrel{d}{=}\sum_{j=1}^{k}j\lambda_{j} Y_{\alpha}(1)\left(\liminf_{t\rightarrow\infty}\frac{D_{f}(t)}{g(t)}\right)^{\alpha},\ \ (\text{using}\ (\ref{limit}))\\
&\stackrel{d}{=}\sum_{j=1}^{k}j\lambda_{j} Y_{\alpha}(1)\theta^\alpha\left(1-\theta\right)^{\alpha(1-\theta)/\theta},
\end{align*}
where the fact that $D_f(t)\to\infty$ as $t\to\infty$ a.s., is used in the penultimate step. Also, the last step follows from the following law of iterated logarithm of L\'evy subordinator (see Bertoin (1996), Theorem 14, p. 92):
\begin{equation*}
\liminf_{t\to\infty}\frac{D_{f}(t)}{g(t)}=\theta(1-\theta)^{(1-\theta)/\theta},\ \ \text{a.s.}
\end{equation*}
This completes the proof.
\end{proof}
\begin{remark}
On substituting $k=1$ in (\ref{lil}), we get the law of iterated logarithm for TCFPP-I (see Maheshwari and Vellaisamy (2019), Theorem 3.5). Also, on taking $\lambda_{j}=\lambda$, $j=1,2,\dots,k$  in (\ref{lil}), we get the law of iterated logarithm for a time-changed version of the TFPPoK as follows:
	\begin{equation*}
	\liminf_{t\rightarrow\infty}\frac{\mathcal{Z}_{f}^{\alpha}(t)}{(g(t))^{\alpha}}\stackrel{d}{=}\frac{k(k+1)}{2} \lambda Y_{\alpha}(1)\theta^\alpha\left(1-\theta\right)^{\alpha(1-\theta)/\theta}.
	\end{equation*}
Moreover, on substituting $\lambda_{j}=\lambda(1-\rho)\rho^{j-1}/(1-\rho^{k})$, $0\leq\rho<1$, $j=1,2,\dots,k$ in (\ref{lil}), we get the law of iterated logarithm for a time-changed version of FPAPoK as follows:
\begin{equation*}
\liminf_{t\rightarrow\infty}\frac{\mathcal{Z}_{f}^{\alpha}(t)}{(g(t))^{\alpha}}\stackrel{d}{=}\frac{\lambda}{1-\rho^{k}}\left(1+\rho+\dots+\rho^{k-1}-k\rho^{k}\right) Y_{\alpha}(1)\theta^\alpha\left(1-\theta\right)^{\alpha(1-\theta)/\theta}.
\end{equation*}
\end{remark}
Orsingher and Toaldo (2015) showed that the Poisson process time-changed by a L\'evy subordinator can be obtained as the limit of a suitable compound Poisson process. A similar result holds true for TCGCP-I.
\begin{theorem}
Let $m$ be a fixed positive integer and $\{X_{j}\}_{j\geq1}$ be a sequence of independent and identically distributed random variables such that
\begin{equation*}
\mathrm{Pr}\{X_{1}=n\}=\frac{1}{u(m)}\int_{0}^{\infty}\mathrm{Pr}\{M(s)= n\}\mu(\mathrm{d}s),\ \ n\ge m,
\end{equation*}
where $u(m)=\displaystyle\int_{0}^{\infty}\mathrm{Pr}\{M(s)\ge m\}\mu(\mathrm{d}s)$. Then, for $t>0$, we have
\begin{equation*}
\lim_{m\to 0}Z_m(t)\stackrel{d}{=}\mathcal{Z}_f(t),
\end{equation*}
where $Z_m(t)=X_{1}+X_2+\cdots+X_{N(tu(m)/\Lambda)}$.
\end{theorem}
\begin{proof}
The pgf of $Z_{m}(t)$ can be written as 
\begin{align*}
\mathbb{E}\left(u^{Z_{m}(t)}\right)&=\exp\left({-tu(m)\left(1-\mathbb{E}(u^{X_1})\right)}\right)\\
&=\exp \left(-tu(m)\sum_{n=m}^{\infty}(1-u^{n})\mathrm{Pr}\{X_1=n\}\right)\\
&=\exp \left(-tu(m)\sum_{n=m}^{\infty}(1-u^{n})\frac{1}{u(m)}\int_{0}^{\infty}\mathrm{Pr}\{M(s)= n\}\mu(\mathrm{d}s)\right)\\
&=\exp \left(-t\int_{0}^{\infty}\sum_{n=m}^{\infty}(1-u^{n})\mathrm{Pr}\{M(s)= n\}\mu(\mathrm{d}s)\right).
\end{align*}
On letting $m\to 0$, we get
\begin{align*}
\lim\limits_{m\to 0}\mathbb{E}\left(u^{Z_{m}(t)}\right)&=\exp \left(-t\int_{0}^{\infty}\sum_{n=0}^{\infty}(1-u^{n})\mathrm{Pr}\{M(s)= n\}\mu(\mathrm{d}s)\right)\\
&=\exp \left(-t\int_{0}^{\infty}\left(1-e^{-s\sum_{j=1}^{k}\lambda_{j}(1-u^{j})}\right)\mu(\mathrm{d}s)\right)\\
&=\exp\left(-tf\left(\sum_{j=1}^{k}\lambda_{j}(1-u^{j})\right)\right).
\end{align*}
This completes the proof.
\end{proof}
\subsection{Dependence structure of TCGFCP-I}\label{section4.1}
Let $0<s\le t<\infty$ and assume that
 \begin{equation*}
l_{1}=r_{1}/\Gamma(\alpha+1),\ \  l_{2}=r_{2}/\Gamma(\alpha+1),\ \ d=\alpha l_{1}^{2} B(\alpha,\alpha+1),
\end{equation*}
 where $r_{1}$ and $r_{2}$ are given in (\ref{r1r2}). From (\ref{es}), the mean of TCGFCP-I is obtained as follows:
\begin{equation*}
\mathbb{E}\left(\mathcal{Z}_{f}^{\alpha}(t)\right)=\mathbb{E}\left(\mathbb{E}\left(M^{\alpha}(D_{f}(t))|D_{f}(t)\right)\right)=l_{1}\mathbb{E}\left(D_{f}^{\alpha}(t)\right).	
\end{equation*}
On substituting  (\ref{meani}) and (\ref{covin}) in (\ref{cs}), we get
\begin{equation*}
\mathbb{E}\left(M^{\alpha}(s)M^{\alpha}(t)\right)=l_{2}s^{\alpha}+ds^{2\alpha}+\alpha l_{1}^{2}t^{2\alpha}B(\alpha,\alpha+1;s/t).
\end{equation*}
Thus,
\begin{align*}
\mathbb{E}\left(\mathcal{Z}_{f}^{\alpha}(s)\mathcal{Z}_{f}^{\alpha}(t)\right)&=\mathbb{E}\left(\mathbb{E}\left(M^{\alpha}(D_{f}(s))M^{\alpha}(D_{f}(t))|D_{f}(s),D_{f}(t)\right)\right)\\
&=l_{2}\mathbb{E}\left(D_{f}^{\alpha}(s)\right)+ d\mathbb{E}\left(D_{f}^{2\alpha}(s)\right) +\alpha l_{1}^{2}  \mathbb{E}\left(D_{f}^{2\alpha}(t)B\left(\alpha,\alpha+1;D_{f}(s)/D_{f}(t)\right)\right).
\end{align*}
Hence, the covariance of TCGFCP-I is given by
\begin{align}\label{cov}
\operatorname{Cov}\left(\mathcal{Z}_{f}^{\alpha}(s),\mathcal{Z}_{f}^{\alpha}(t)\right)&=l_{2}\mathbb{E}\left(D_{f}^{\alpha}(s)\right)+ d\mathbb{E}\left(D_{f}^{2\alpha}(s)\right)-l_{1}^{2}\mathbb{E}\left(D_{f}^{\alpha}(s)\right)\mathbb{E}\left(D_{f}^{\alpha}(t)\right)\nonumber\\
&\ \ 
+ \alpha l_{1}^{2}  \mathbb{E}\left(D_{f}^{2\alpha}(t)B\left(\alpha,\alpha+1;D_{f}(s)/D_{f}(t)\right)\right).
\end{align}
On putting $s=t$ in (\ref{cov}), we get its  variance as follows:
\begin{equation}\label{var}
\operatorname{Var}\left(\mathcal{Z}_{f}^{\alpha}(t)\right)=\mathbb{E}\left(D_{f}^{\alpha}(t)\right)\left(l_{2}-l_{1}^{2}\mathbb{E}\left(D_{f}^{\alpha}(t)\right)\right)+2d\mathbb{E}\left(D_{f}^{2\alpha}(t)\right).
\end{equation}

Next, we show that the TCGFCP-I has the LRD property provided the associated L\'evy subordinator satisfies certain asymptotic conditions.

\begin{theorem}
Let $\mathbb{E}\left(D_f^{i\alpha}(t)\right)\sim k_it^{i\theta}$ for $i=1,2$
such that  $0< \theta <1$, $k_1>0$ and $k_2\geq k_1^2$. Then, the TCGFCP-I  exhibits the LRD property.
\end{theorem}	
\begin{proof}
For fixed $s$ and large $t$, the following asymptotic result holds (see Maheshwari and Vellaisamy (2019), Theorem 3.3):
\begin{equation*}
\alpha\mathbb{E}\left(D_f^{2\alpha}(t)B\left(\alpha,\alpha+1;D_f(s)/D_f(t)\right)\right)\sim \mathbb{E}\left(D_f^{\alpha}(s)\right)\E\left(D_f^{\alpha}(t-s)\right).
\end{equation*}
On using it in (\ref{cov}), we get
\begin{align*}
\operatorname{Cov}\left(\mathcal{Z}_{f}^{\alpha}(s),\mathcal{Z}_{f}^{\alpha}(t)\right)&\sim l_{2}\mathbb{E}\left(D_{f}^{\alpha}(s)\right)+ d\mathbb{E}\left(D_{f}^{2\alpha}(s)\right)-l_{1}^{2}\mathbb{E}\left(D_{f}^{\alpha}(s)\right)\mathbb{E}\left(D_{f}^{\alpha}(t)\right)\\
&\ \ \ 
+l_{1}^{2}  \mathbb{E}\left(D_f^{\alpha}(s)\right)\E\left(D_f^{\alpha}(t-s)\right)\\
&\sim l_{2}\mathbb{E}\left(D_{f}^{\alpha}(s)\right)+ d\mathbb{E}\left(D_{f}^{2\alpha}(s)\right)-l_{1}^{2}  \mathbb{E}\left(D_f^{\alpha}(s)\right)k_1(t^\theta-(t-s)^\theta)\\
&\sim l_{2}\mathbb{E}\left(D_{f}^{\alpha}(s)\right)+ d\mathbb{E}\left(D_{f}^{2\alpha}(s)\right)-l_{1}^{2}  \mathbb{E}\left(D_f^{\alpha}(s)\right)k_1s\theta t^{\theta-1},
\end{align*}	
where we have used $\mathbb{E}\left(D_f^{\alpha}(t)\right)\sim k_1t^{\theta}$ in the penultimate step.
	
Similarly, from (\ref{var}), we get
\begin{align*}
\operatorname{Var}\left(\mathcal{Z}_{f}^{\alpha}(t)\right)&\sim l_{2}k_1t^\theta- l_{1}^{2}k^{2}_1t^{2\theta} +2dk_2t^{2\theta}\\
&\sim \left(2dk_2-k^{2}_1l_{1}^{2}\right)t^{2\theta}.
\end{align*}
\noindent For large $t$, we have
\begin{align*}
\operatorname{Corr}\left(\mathcal{Z}_{f}^{\alpha}(s),\mathcal{Z}_{f}^{\alpha}(t)\right)&\sim\frac{l_{2}\mathbb{E}\left(D_{f}^{\alpha}(s)\right)+ d\mathbb{E}\left(D_{f}^{2\alpha}(s)\right)-l_{1}^{2}  \mathbb{E}\left(D_f^{\alpha}(s)\right)k_1s\theta t^{\theta-1}}{\sqrt{\operatorname{Var}\left(\mathcal{Z}_{f}^{\alpha}(s)\right)}\sqrt{\left(2dk_2-k^{2}_1l_{1}^{2}\right)t^{2\theta}}}\\
&\sim c_{1}(s)t^{-\theta},  		
\end{align*}
\noindent where
\begin{equation*}
c_{1}(s)=\frac{l_{2}\mathbb{E}\left(D_{f}^{\alpha}(s)\right)+ d\mathbb{E}\left(D_{f}^{2\alpha}(s)\right)}{\sqrt{\operatorname{Var}\left(\mathcal{Z}_{f}^{\alpha}(s)\right)\left(2dk_2-k^{2}_1l_{1}^{2}\right)}}.
\end{equation*}		
As $0< \theta <1$, the proof follows.
\end{proof}
\begin{remark}
Along the similar lines it can be shown that TCGCP-I exhibits the LRD property.
\end{remark}

\subsection{Some special cases of the TCGCP-I}
Here, we discuss three special cases of the TCGCP-I by taking three specific L\'evy subordinators, namely, the gamma subordinator, the tempered stable subordinator (TSS) and the inverse Gaussian subordinator (IGS) as a time-change component in the GCP.
\subsubsection{GCP time-changed by gamma subordinator}
The pdf $g(x,t)$ of a gamma subordinator $\{Z(t)\}_{t\ge0}$ is given by
\begin{equation*}
g(x,t)=\frac{a^{bt}}{\Gamma(bt)}x^{bt-1}e^{-ax},\ \ x>0,
\end{equation*}
where $a>0$ and $b>0$. Its associated Bern\v stein function is $f_{1}(s)=b\log(1+s/a)$, $s>0$ (see Applebaum (2009), p. 55).

On taking $f_1$ as the Bern\v stein function in (\ref{zft}), we get the GCP time-changed by an independent gamma subordinator as 
\begin{equation}\label{gam}
\mathcal{Z}_{f_{1}}(t)\coloneqq M
(Z(t)),\ \ t\ge0.
\end{equation}

On using (\ref{bgd43}), the distribution of its jumps is obtained in the following form:
\begin{equation}\label{mlkjuy7}
\mathrm{Pr}\{\mathcal{Z}_{f_{1}}(h)=n\}=\begin{cases*}
1-bh\log(1+\Lambda/a)+o(h),\ \ n=0,\\
\displaystyle-bh\sum_{ \Omega(k,n)}\frac{(-1)^{s_{k}-1}(s_{k}-1)!}{(a+\Lambda)^{s_{k}}}\prod_{j=1}^{k}\frac{(-\lambda_{j})^{x_{j}}}{x_{j}!}+o(h),\ \ n\ge1.
\end{cases*}
\end{equation}
\begin{remark}
On taking $k=a=b=1$ in (\ref{mlkjuy7}), we get the distribution of jumps of gamma-Poisson process (see Orsingher and Toaldo (2015), Eq. (4.16)).
\end{remark}

From (\ref{pgff}), the pgf of $\mathcal{Z}_{f_{1}}(t)$ is given by
\begin{equation*}
G_{f_{1}}(u,t)=\left(1+\frac{1}{a}\sum_{j=1}^{k}\lambda_{j}(1-u^{j})\right)^{-bt}.
\end{equation*}
\begin{proposition}\label{pwsq23}
The L\'evy measure of $\mathcal{Z}_{f_{1}}(t)$ is given by
\begin{equation}\label{bc56}
\Pi_{f_{1}}(\mathrm{d}x)=\sum_{n=1}^{\infty}\sum_{ \Omega(k,n)}\prod_{j=1}^{k}\frac{\lambda_{j}^{x_{j}}}{x_{j}!}\frac{b\Gamma(s_{k})}{(\Lambda+a)^{s_{k}}}\delta_{n}(\mathrm{d}x).
\end{equation}
\end{proposition}
\begin{proof}
The L\'evy measure for gamma subordinator is given by $\mu_{Z}(\mathrm{d}s
)=bs^{-1}e^{-as}\mathrm{d}s$. Using a result (see Sato (1999), Theorem 30.1, p. 197), the  L\'evy measure of $\mathcal{Z}_{f_{1}}(t)$ is obtained as follows:
\begin{align*}
\Pi_{f_{1}}(\mathrm{d}x)&=\int_{0}^{\infty}\sum_{n=1}^{\infty}p(n,s)\delta_{n}(\mathrm{d}x)\mu_{Z}(\mathrm{d}s)\\	&=\sum_{n=1}^{\infty}\sum_{ \Omega(k,n)}\prod_{j=1}^{k}\frac{\lambda_{j}^{x_{j}}}{x_{j}!}b\delta_{n}(\mathrm{d}x)\int_{0}^{\infty}s^{s_{k}-1}e^{-(\Lambda+a)s}\mathrm{d}s,
\end{align*}
where we have used (\ref{p(n,t)}). This gives the required result.	
\end{proof}
\begin{remark}
On taking $k=1$ in (\ref{bc56}), we get
\begin{equation*}
\Pi_{f_{1}}(\mathrm{d}x)\Big|_{k=1}=	\sum_{n=1}^{\infty}\frac{b}{n}\left(\frac{\lambda_1}{\lambda_1+a}\right)^{n}\delta_{n}(\mathrm{d}x),
\end{equation*}
which is the L\'evy measure of negative binomial process (see Beghin and Vellaisamy (2018)).
\end{remark}
\begin{proposition}
Let $\gamma\ge1$ and $\kappa(x)\coloneqq\Gamma^{\prime}(x)/\Gamma(x)$ be the digamma function. Then, the pmf $p_{f_{1}}(n,t)=\mathrm{Pr}\{\mathcal{Z}_{f_{1}}(t)=n\}$, $n\geq0$ solves the following equation:
\begin{equation*}
D_t^{\gamma}p_{f_{1}}(n,t)=bD_t^{\gamma-1}\left(\log (a)-\kappa(bt)\right)p_{f_{1}}(n,t)+b\int_{0}^{\infty}p(n,x)\log (x)D_t^{\gamma-1}g(x,t)\,\mathrm{d}x,
\end{equation*}
where $D_t^{\gamma}$ is the R-L fractional derivative defined in (\ref{RLd}).
\end{proposition}
\begin{proof}
From (\ref{gam}), we have
\begin{equation}\label{qaza122}
p_{f_{1}}(n,t)=\int_{0}^{\infty}p(n,x)g(x,t)\,\mathrm{d}x.
\end{equation}
The following result holds for the pdf of gamma subordinator (see Beghin and Vellaisamy (2018), Lemma 2.2):
\begin{align*}
D_t^{\gamma}g(x,t)&=bD_t^{\gamma-1}\left(\log(a x)-\kappa(bt)\right)g(x,t), \ \ \ x>0,\\
g(x,0)&=0.	
\end{align*}
Taking the R-L fractional derivative in (\ref{qaza122}) and using the above result, we get
\begin{align*}
D_t^{\gamma}p_{f_{1}}(n,t)&=\int_{0}^{\infty}p(n,x)D_t^{\gamma}g(x,t)\,\mathrm{d}x\\
&=b\int_{0}^{\infty}p(n,x)D_t^{\gamma-1}\left(\log(a x)-\kappa(bt)\right)g(x,t)\,\mathrm{d}x\\
&=bD_t^{\gamma-1}\log (a) \int_{0}^{\infty}p(n,x)g(x,t)\,\mathrm{d}x+b \int_{0}^{\infty} p(n,x)\log (x)D_t^{\gamma-1}g(x,t)\,\mathrm{d}x\\
&\ \ - bD_t^{\gamma-1}\kappa(bt)\int_{0}^{\infty}p(n,x)g(x,t)\,\mathrm{d}x.
\end{align*}
The proof follows on using (\ref{qaza122}).
\end{proof}
\subsubsection{GCP time-changed by tempered stable subordinator}
Let $0<\theta<1$ be the stability index and $\eta>0$ be the tempering parameter of a TSS $\{\mathscr{D}_{\eta,\theta}(t)\}_{t\ge0}$. Its associated Bern\v stein function $f_{2}(s)$ is given by 
\begin{equation}\label{bs1}
f_{2}(s)=(\eta+s)^{\theta}-\eta^{\theta},\ s>0.
\end{equation}

On taking $f_2$ as the Bern\v stein function in (\ref{zft}), we get the GCP time-changed by an independent TSS as 
\begin{equation}\label{tss}
\mathcal{Z}_{f_{2}}(t)\coloneqq M(\mathscr{D}_{\eta,\theta}(t)),\ \ t\ge0.
\end{equation}	

On using (\ref{bgd43}), the distribution of its jumps is obtained in the following form: 
\begin{equation}\label{tgfd45}
\mathrm{Pr}\{\mathcal{Z}_{f_{2}}(h)=n\}=\begin{cases*}
1-h\left((\eta+\Lambda)^{\theta}-\eta^{\theta}\right)+o(h),\ \ n=0,\\
\displaystyle- h\sum_{ \Omega(k,n)}(\theta)_{s_{k}}(\eta+\Lambda)^{\theta-s_{k}}\prod_{j=1}^{k}\frac{(-\lambda_{j})^{x_{j}}}{x_{j}!}+o(h),\ \ n\ge1,
\end{cases*}
\end{equation}
where $(\theta)_{s_{k}}=\theta(\theta-1)\cdots(\theta-s_k+1)$.
\begin{remark}
On taking $k=1$ in (\ref{tgfd45}), we get the distribution of jumps of relativistic Poisson process (see Orsingher and Toaldo (2015), Eq. (4.11)).
\end{remark}

From (\ref{pgff}), its pgf is given by
\begin{equation*}
G_{f_{2}}(u,t)=\exp\left\{-t\left(\left(\eta+\sum_{j=1}^{k}\lambda_{j}(1-u^{j})\right)^{\theta}-\eta^{\theta}\right)\right\}.
\end{equation*}
\begin{proposition}\label{pwsq}
The L\'evy measure of $\mathcal{Z}_{f_{2}}(t)$ is given by
\begin{equation*}
\Pi_{f_{2}}(\mathrm{d}x)=	\frac{\theta}{\Gamma(1-\theta)}\sum_{n=1}^{\infty}\sum_{ \Omega(k,n)}\prod_{j=1}^{k}\frac{\lambda_{j}^{x_{j}}}{x_{j}!}\frac{\Gamma(s_{k}-\theta)}{(\Lambda+\eta)^{s_{k}-\theta}}\delta_{n}(\mathrm{d}x).
\end{equation*}
\end{proposition}
The proof of Proposition \ref{pwsq} follows similar lines to that of Proposition \ref{pwsq23} by using the L\'evy measure of TSS, that is,  $\mu_{\mathscr{D}_{\eta,\theta}}(\mathrm{d}s)=\theta s^{-\theta-1}e^{-\eta s}\mathrm{d}s/\Gamma(1-\theta)$. 

\begin{proposition}
The pmf $p_{f_{2}}(n,t)=\mathrm{Pr}\{\mathcal{Z}_{f_{2}}(t)=n\}$ satisfies the following system of differential equations:
\begin{equation*}
\left(\eta^{\theta}-\frac{\mathrm{d}}{\mathrm{d} t}\right)^{1/\theta}p_{f_{2}}(n,t)=(\eta +\Lambda) p_{f_{2}}(n,t)-\sum_{j=1}^{\min\{n,k\}}\lambda_{j}p_{f_{2}}(n-j,t),\ \ n\ge0.
\end{equation*}
\end{proposition}
\begin{proof}
From (\ref{tss}), we have
\begin{equation}\label{wslp09}
p_{f_{2}}(n,t)=\int_{0}^{\infty}p(n,x)	h_{\eta,\theta}(x,t)\,\mathrm{d}x,
\end{equation}
where $h_{\eta,\theta}(\cdot,t)$ is the pdf of TSS. 
	
Let $\delta_0(x)$ denote the Dirac delta function. On using the fact that $\lim_{x\to 0}h_{\eta,\theta}(x,t)=\lim_{x\to \infty}h_{\eta,\theta}(x,t)=0$ and the following result (See Beghin (2015), Eq. (15)):
\begin{equation*}
\frac{\partial}{\partial x}h_{\eta,\theta}(x,t)=-\eta h_{\eta,\theta}(x,t)+\left(\eta^{\theta}-\frac{\partial}{\partial t}\right)^{1/\theta}h_{\eta,\theta}(x,t),
\end{equation*}
with the initial conditions $h_{\eta,\theta}(x,0)=\delta_0(x)$ and $h_{\eta,\theta}(0,t)=0$ in (\ref{wslp09}), we get 
\begin{align*}
\left(\eta^{\theta}-\frac{\mathrm{d}}{\mathrm{d} t}\right)^{1/\theta}p_{f_{2}}(n,t)&=\int_{0}^{\infty}p(n,x)\left(\eta h_{\eta,\theta}(x,t)+\frac{\partial}{\partial x}	h_{\eta,\theta}(x,t)\right)\,\mathrm{d}x\\
&=\eta p_{f_{2}}(n,t)-\int_{0}^{\infty}h_{\eta,\theta}(x,t)\frac{\mathrm{d}}{\mathrm{d} x}p(n,x)	\,\mathrm{d}x\\
&=\eta p_{f_{2}}(n,t)- \int_{0}^{\infty}\left(-\Lambda p(n,x)+	\sum_{j=1}^{\min\{n,k\}}\lambda_{j}p(n-j,x)\right)h_{\eta,\theta}(x,t)\,\mathrm{d}x,
\end{align*}
where we have used (\ref{cre}) with $\alpha=1$ in the last step. The proof follows on using (\ref{wslp09}).
\end{proof}
If $\theta^{-1}=m\ge2$  is an integer then the pmf $p_{f_{2}}(n,t)$ solves 
\begin{equation}\label{2lowsaa}
\sum_{i=1}^{m}(-1)^{i}\binom{m}{i}\eta^{(1-i/m)}\frac{\mathrm{d}^{i}}{\mathrm{d} t^{i}}p_{f_{2}}(n,t)=\Lambda p_{f_{2}}(n,t)-\sum_{j=1}^{\min\{n,k\}}\lambda_{j}p_{f_{2}}(n-j,t).
\end{equation}
Further, on putting $k=1$ in (\ref{2lowsaa}), we get the system of differential equations that governs the state probabilities of Poisson process time-changed by TSS (see Kumar {\it et al.} (2011), Remark 4.1).  

\subsubsection{GCP time-changed by inverse Gaussian subordinator}
Let $\{Y(t)\}_{t\ge0}$ be an IGS whose pdf is given by (see Applebaum (2009), Eq. (1.27))
\begin{equation*}
q(x,t)=(2\pi)^{-1/2}\delta tx^{-3/2}\exp\left\{\delta\gamma t-\frac{1}{2}(\delta^{2}t^{2}x^{-1}+\gamma^{2}x)\right\}, \ \ x>0,
\end{equation*}
where $\delta>0$ and $\gamma>0$. Its associated Bern\v stein function is
\begin{equation}\label{ploiuy67}
f_{3}(s)=\delta\left(\sqrt{2s+\gamma^2}-\gamma\right),\ \ s>0.
\end{equation}

On taking $f_3$ as the Bern\v stein function in (\ref{zft}), we get the GCP time-changed by an independent IGS as 
\begin{equation}\label{ig}
\mathcal{Z}_{f_{3}}(t)\coloneqq M(Y(t)),\ \ t\ge0.
\end{equation}

On using (\ref{bgd43}), the distribution of its jumps is obtained in the  following form:
\begin{equation*}
\mathrm{Pr}\{\mathcal{Z}_{f_{3}}(h)=n\}=\begin{cases*}
1-h\delta\left(\sqrt{2\Lambda+\gamma^2}-\gamma\right)+o(h), \ n=0\\
\displaystyle- \delta h\sum_{ \Omega(k,n)}2^{s_{k}}\left(\frac{1}{2}\right)_{s_{k}}\left(2\Lambda+\gamma^2\right)^{\frac{1}{2}-s_{k}}\prod_{j=1}^{k}\frac{(-\lambda_{j})^{x_{j}}}{x_{j}!}+o(h),\ \ n\ge1.
\end{cases*}
\end{equation*}
where $\left(\frac{1}{2}\right)_{s_{k}}$ denotes the falling factorial.

From (\ref{pgff}), the pgf of $\mathcal{Z}_{f_{3}}(t)$ is given by
\begin{equation*}
G_{f_{3}}(u,t)=\exp\left\{-t\delta\left(\sqrt{2\sum_{j=1}^{k}\lambda_{j}(1-u^{j})+\gamma^{2}}-\gamma\right)\right\}.
\end{equation*}
\begin{proposition}\label{pwsqk87}
The L\'evy measure of $\mathcal{Z}_{f_{3}}(t)$ is given by
\begin{equation*}
\Pi_{f_{3}}(\mathrm{d}x)=	\frac{\delta}{\sqrt{2\pi}}\sum_{n=1}^{\infty}\sum_{ \Omega(k,n)}\prod_{j=1}^{k}\frac{\lambda_{j}^{x_{j}}}{x_{j}!}\frac{\Gamma(s_{k}-1/2)}{(\Lambda+\gamma^{2}/2)^{s_{k}-1/2}}\delta_{n}(\mathrm{d}x).
\end{equation*}
\end{proposition}
The proof of Proposition \ref{pwsqk87} follows similar lines to that of Proposition \ref{pwsq23} by using the L\'evy measure of IGS, that is,  $\mu_{Y}(\mathrm{d}s)=\delta e^{-\gamma^{2}s/2}\mathrm{d}s/\sqrt{2\pi s^{3}}$. 
\begin{proposition}\label{p4.1}
The pmf  $p_{f_{3}}(n,t)=\mathrm{Pr}\{\mathcal{Z}_{f_{3}}(t)=n\}$ satisfies the following system of differential equations:
\begin{align}\label{lkii1}
\left(\frac{\mathrm{d}^2}{\mathrm{d}t^2}-2\delta \gamma \frac{\mathrm{d}}{\mathrm{d}t}\right) p_{f_{3}}(n,t)&=2\delta^2\left(\Lambda p_{f_{3}}(n,t)-\sum_{j=1}^{\min\{n,k\}}\lambda_{j}p_{f_{3}}(n-j,t)\right),\ \ n\ge0.
\end{align}
\end{proposition}
\begin{proof}
From (\ref{ig}), we have
\begin{equation}\label{123}
p_{f_{3}}(n,t)=\int_0^{\infty}p(n,x)q(x,t)\,\mathrm{d}x.
\end{equation}	
On taking derivatives, we get
\begin{equation*}
\frac{\mathrm{d}}{\mathrm{d}t} p_{f_{3}}(n,t) = \int_0^{\infty}p(n,x) \frac{\partial}{\partial t} q(x,t)\,\mathrm{d}x
\end{equation*}
and
\begin{equation*}
\frac{\mathrm{d}^2}{\mathrm{d} t^2} p_{f_{3}}(n,t) = \int_0^{\infty} p(n,x) \frac{\partial^2}{\partial t^2} q(x,t)\,\mathrm{d}x. 
\end{equation*}
On using the fact that $\lim_{x\to\infty}q(x,t)=\lim_{x\to 0}q(x,t)=0$ and the following result for the pdf of IGS  (see Vellaisamy and Kumar (2018), Eq. (3.3)):
\begin{equation*}
\frac{\partial^2}{\partial t^2}q(x,t)-2\delta \gamma \frac{\partial}{\partial t}q(x,t) = 2\delta^2\frac{\partial}{\partial x}q(x,t)
\end{equation*}
in (\ref{123}), we get
\begin{align*}
\left(\frac{\mathrm{d}^2}{\mathrm{d} t^2}-2\delta \gamma \frac{\mathrm{d}}{\mathrm{d} t}\right) p_{f_{3}}(n,t)
&=\int_0^{\infty}p(n,x) \left(\frac{\partial^2}{\partial t^2}-2\delta \gamma \frac{\partial}{\partial t}\right) q(x,t)\,\mathrm{d}x\\
&= 2\delta^2 \int_0^{\infty} p(n,x)\frac{\partial}{\partial x} q(x,t)\,\mathrm{d}x\\ 
&= -2\delta^2 \int_0^{\infty} q(x,t)\frac{\mathrm{d}}{\mathrm{d}x} p(n,x)\, \mathrm{d}x\label{jahxa11}\\
&=-2\delta^2\int_0^{\infty}\bigg(-\Lambda p(n,x)+\sum_{j=1}^{\min\{n,k\}}\lambda_{j}p(n-j,x)\bigg)q(x,t)\,\mathrm{d}x,
\end{align*}
where we have used (\ref{cre}) with $\alpha=1$. The proof is complete on using (\ref{123}). 
\end{proof}
\begin{remark}
Taking $\lambda_{j}=\lambda$ for all $j=1,2,\ldots,k$ in (\ref{lkii1}), we get the system of differential equations that governs the state probabilities of a time-changed PPoK (see Sengar {\it et al.} (2020), Theorem 5.1). For $k=1$ in (\ref{lkii1}), we get the corresponding result for the Poisson process time-changed by IGS (see Kumar {\it et al.} (2011), Proposition 2.1).
\end{remark}

\section{GFCP time-changed by inverse subordinator}
Here, we consider another time-changed version of the GFCP by using the inverse subordinator. The first passage time  of L\'evy subordinator $\{D_f (t)\}_{t\ge0}$ is called the inverse subordinator. It is defined as
\begin{equation*}
H_f (t)\coloneqq\inf\{r\ge0: D_f (r)> t\}, \ \ t\ge0.
\end{equation*}
Note that $\mathbb{E}\left(H^r_f(t)\right)<\infty$ for all $r>0$ (see Aletti {\it et al.} (2018), Section 2.1). 

We define a time-changed version of the GFCP by time-changing it with an independent inverse subordinator as follows:
\begin{equation*}
\bar{\mathcal{Z}}^{\alpha}_{f}(t)\coloneqq M^{\alpha}(H_f (t)),\ \ t\ge0.
\end{equation*}
We call the process $\{\bar{\mathcal{Z}}^{\alpha}_{f}(t)\}_{t\ge0}$ as the time-changed  generalized fractional counting process-II (TCGFCP-II).

For $\alpha=1$, the TCGFCP-II reduces to a time-changed version of the GCP, namely, the time-changed generalized counting process-II (TCGCP-II), that is, 
\begin{equation}\label{loppp1}
\bar{\mathcal{Z}}_{f}(t)\coloneqq M(H_f (t)),\ \ t\ge0.
\end{equation}
The pmf $\bar{p}_{f}(n,t)=\mathrm{Pr}\{\bar{\mathcal{Z}}_{f}(t)=n\}$ of TCGCP-II is given by 
\begin{equation*}
\bar{p}_{f}(n,t)=\sum_{ \Omega(k,n)}\prod_{j=1}^{k}\frac{\lambda_{j}^{x_{j}}}{x_{j}!}\mathbb{E}\left(e^{-\Lambda H_{f}(t)}H_{f}^{s_{k}}(t)\right),\ \ n\geq0.
\end{equation*}

The proof of the above result follows similar lines to that of Theorem \ref{thm1}.

Let $l_{1}$, $l_{2}$ and $d$ be as given in Section \ref{section4.1}. The mean, variance and covariance of TCGFCP-II are given by\\
	
\noindent $(i)$ $\mathbb{E}\left(\bar{\mathcal{Z}}^{\alpha}_{f}(t)\right)=l_{1}\mathbb{E}\left(H_{f}^{\alpha}(t)\right)$,\vspace*{.1cm}\\
\noindent $(ii)$ $\operatorname{Var}\left(\bar{\mathcal{Z}}^{\alpha}_{f}(t)\right)=\mathbb{E}\left(H_{f}^{\alpha}(t)\right)\left(l_{2}-l_{1}^{2}\mathbb{E}\left(H_{f}^{\alpha}(t)\right)\right)+2d\mathbb{E}\left(H_{f}^{2\alpha}(t)\right)$,\vspace*{.1cm}\\
\noindent $(iii)$ $ \operatorname{Cov}\left(\bar{\mathcal{Z}}^{\alpha}_{f}(s),\bar{\mathcal{Z}}^{\alpha
}_{f}(t)\right)=l_{2}\mathbb{E}\left(H_{f}^{\alpha}(s)\right)+ d\mathbb{E}\left(H_{f}^{2\alpha}(s)\right)-l_{1}^{2}\mathbb{E}\left(H_{f}^{\alpha}(s)\right)\mathbb{E}\left(H_{f}^{\alpha}(t)\right)$\vspace*{.1cm}\\
$\hspace*{4.8cm}
+ \alpha l_{1}^{2}  \mathbb{E}\left(H_{f}^{2\alpha}(t)B\left(\alpha,\alpha+1;H_{f}(s)/H_{f}(t)\right)\right)$, $0<s\le t$.

The proof of  $(i)$-$(iii)$ follows similar lines to the corresponding results of TCGFCP-I (see Section \ref{section4.1}). Thus, the proofs are omitted. 

Next we discuss two particular cases of the TCGCP-II.
\subsection{GCP time-changed by the inverse TSS}
The inverse TSS $\{\mathscr{L}_{\eta,\theta}(t)\}_{t\ge0}$ is defined as the first passage time of TSS $\{\mathscr{D}_{\eta,\theta}(t)\}_{t\ge0}$, $0<\theta<1$, $\eta>0$, that is,
\begin{equation*}
\mathscr{L}_{\eta,\theta}(t)\coloneqq \inf \{ r\ge 0 : \mathscr{D}_{\eta,\theta}(r)>t\},\ \ t\ge0.
\end{equation*}
In (\ref{loppp1}), if we choose the Bern\v stein function $f_{2}$ which is given in (\ref{bs1}) then we get the GCP time-changed by an independent inverse TSS. Thus,
\begin{equation}\label{inv1}
\bar{\mathcal{Z}}_{f_{2}}(t)\coloneqq M(\mathscr{L}_{\eta,\theta} (t)),\ \ t\ge0.
\end{equation}
\begin{proposition}
The pmf $\bar{p}_{f_{2}}(n,t)=\mathrm{Pr}\{\bar{\mathcal{Z}}_{f_2}(t)=n\}$ solves the following system of  differential equations:
\begin{align*}
\left(\eta+\frac{\mathrm{d}}{\mathrm{d} t}\right)^{\theta}\bar{p}_{f_{2}}(n,t)&=\eta^{\theta}\bar{p}_{f_{2}}(n,t)-t^{-\theta}E^{1-\theta}_{1,1-\theta}(-\eta t)p(n,0)+p(n,x)l_{\eta,\theta}(x,t)\big|_{x=0}\\
&\ \ -\Lambda\bar{p}_{f_{2}}(n,t)+\sum_{j=1}^{\min\{n,k\}}\lambda_{j}\bar{p}_{f_{2}}(n-j,t),\ \ n\ge 0.
\end{align*}
\end{proposition}
\begin{proof}
From (\ref{inv1}), we have
\begin{equation}\label{gfdt}
\bar{p}_{f_{2}}(n,t)=\int_{0}^{\infty}p(n,x)l_{\eta,\theta}(x,t)\,\mathrm{d}x,
\end{equation}
where $l_{\eta,\theta}(\cdot,t)$ is the pdf of $\mathscr{L}_{\eta,\theta} (t)$.
	
The following result holds (See Kumar {\it et al.} (2019), Eq. (25)):
\begin{equation}\label{qmk11}
\frac{\partial}{\partial x}l_{\eta,\theta}(x,t)=-\left(\eta+\frac{\partial}{\partial t}\right)^{\theta}l_{\eta,\theta}(x,t)+\eta^{\theta} l_{\eta,\theta}(x,t)-t^{-\theta}E^{1-\theta}_{1,1-\theta}(-\eta t)\delta_0(x),
\end{equation}
where $\delta_0(x)=l_{\eta,\theta}(x,0)$. From (\ref{gfdt}) and (\ref{qmk11}), we get
\begin{align*}
\left(\eta+\frac{\mathrm{d}}{\mathrm{d} t}\right)^{\theta}\bar{p}_{f_{2}}(n,t)&=\int_{0}^{\infty}p(n,x)\left(\eta^{\theta} l_{\eta,\theta}(x,t)-t^{-\theta}E^{1-\theta}_{1,1-\theta}(-\eta t)\delta_0(x)-\frac{\partial}{\partial x}l_{\eta,\theta}(x,t)\right)\,\mathrm{d}x\\
&=\eta^{\theta}\bar{p}_{f_{2}}(n,t)-t^{-\theta}E^{1-\theta}_{1,1-\theta}(-\eta t)\int_{0}^{\infty}p(n,x)\delta_0(x)\,\mathrm{d}x\\
&\ \ +p(n,x)l_{\eta,\theta}(x,t)\big|_{x=0}+\int_{0}^{\infty}l_{\eta,\theta}(x,t)\frac{\mathrm{d}}{\mathrm{d}x}p(n,x)\,\mathrm{d}x\\
&=\eta^{\theta}\bar{p}_{f_{2}}(n,t)-t^{-\theta}E^{1-\theta}_{1,1-\theta}(-\eta t)p(n,0)+p(n,x)l_{\eta,\theta}(x,t)\big|_{x=0}\\
&\ \ +\int_{0}^{\infty}\left(-\Lambda p(n,x)+	\sum_{j=1}^{\min\{n,k\}}\lambda_{j}p(n-j,x)\right)l_{\eta,\theta}(x,t)\,\mathrm{d}x,
\end{align*}
where we have used $\lim_{x\to \infty}l_{\eta,\theta}(x,t)=0$ (see Alrawashdeh {\it et al.} (2017), Lemma 4.6). On using (\ref{gfdt}), we get the required result.
\end{proof}
\subsection{GCP time-changed by the first passage time of IGS}
The first passage time $\{H(t)\}_{t\ge0}$ of the IGS $\{Y(t)\}_{t\geq0}$ is  defined as
\begin{equation*}
H(t)\coloneqq \inf \{ r\ge 0 : Y(r)>t\},\ \ t\geq0.
\end{equation*}
The function $f_{3}$ given in (\ref{ploiuy67}) is the associated Bern\v stein function for an IGS. In (\ref{loppp1}), if we choose the Bern\v stein function $f_{3}$ then we get the GCP time-changed by an independent $\{H(t)\}_{t\ge0}$, that is,
\begin{equation}\label{inv}
\bar{\mathcal{Z}}_{f_{3}}(t)\coloneqq M(H (t)),\ \ t\ge0.
\end{equation}
\begin{proposition}
The pmf $\bar{p}_{f_3}(n,t)=\mathrm{Pr}\{\bar{\mathcal{Z}}_{f_{3}}(t)=n\}$, $n\geq0$ solves the following system of differential equations:
\begin{equation*}
\delta\left(\gamma^{2}+2\frac{\mathrm{d}}{\mathrm{d} t}\right)^{1/2}\bar{p}_{f_{3}}(n,t)=\left(\delta\gamma-\Lambda\right)\bar{p}_{f_{3}}(n,t)+\sum_{j=1}^{\min\{n,k\}}\lambda_{j}\bar{p}_{f_{3}}(n-j,t)-\delta\gamma \mathrm{Erf}\left(\gamma\sqrt{t/2}\right)p(n,0),
\end{equation*}
where $\mathrm{Erf}(\cdot)$ is the error function.
\end{proposition}
\begin{proof}
Let $h(\cdot,t)$ be the pdf of $H(t)$. From (\ref{inv}), we have
\begin{equation}\label{gfdt1}
\bar{p}_{f_{3}}(n,t)=\int_{0}^{\infty}p(n,x)h(x,t)\,\mathrm{d}x.
\end{equation}
	
On using the following result in (\ref{gfdt1}) (see Wyloma\'nska {\it et al.} (2016), Eq. (2.22)):
\begin{equation*}
\frac{\partial}{\partial x}h(x,t)=-\delta\left(\gamma^{2}+2\frac{\partial}{\partial t}\right)^{1/2}h(x,t)+\delta\gamma h(x,t)-\delta\sqrt{2/\pi t}e^{-\gamma^{2}t/2}\delta_0(x),
\end{equation*}	
where the initial condition is $h(x,0)=\delta_0(x)$, we get
\begin{align*}
\delta\left(\gamma^{2}+2\frac{\mathrm{d}}{\mathrm{d} t}\right)^{1/2}\bar{p}_{f_{3}}(n,t)&=\int_0^{\infty} p(n,x)\left(\delta\gamma h(x,t)-\delta\sqrt{2/\pi t}e^{-\gamma^{2}t/2}\delta_0(x)-\frac{\partial}{\partial x}h(x,t)\right)\mathrm{d}x\\
&=\delta\gamma\bar{p}_{f_{3}}(n,t)-\delta\sqrt{2/\pi t}e^{-\gamma^{2}t/2}p(n,0)\\
&\ \ +p(n,0)h(0,t)+\int_0^{\infty}h(x,t)\frac{\mathrm{d}}{\mathrm{d}x}p(n,x)\mathrm{d}x\\
&=\delta\gamma\bar{p}_{f_{3}}(n,t)-\delta\sqrt{2/\pi t}e^{-\gamma^{2}t/2}p(n,0)+p(n,0)h(0,t)\\
&\ \ +\int_0^{\infty}\left(-\Lambda p(n,x)+	\sum_{j=1}^{\min\{n,k\}}\lambda_{j}p(n-j,x)\right)h(x,t)\mathrm{d}x\\
&=\delta\gamma\bar{p}_{f_{3}}(n,t)-\delta\sqrt{2/\pi t}e^{-\gamma^{2}t/2}p(n,0)+p(n,0)h(0,t)\\
&\ \ -\Lambda\bar{p}_{f_{3}}(n,t)+\sum_{j=1}^{\min\{n,k\}}\lambda_{j}\bar{p}_{f_{3}}(n-j,t),
\end{align*}	
where we have used (\ref{gfdt1}). The proof follows on using the following result (see Vellaisamy and Kumar (2018), Proposition 2.2):
\begin{equation*}
\lim\limits_{x\to0}h(x,t)=h(0,t)=\delta e^{-\gamma^{2}t/2}\left(\sqrt{2/\pi t}-\gamma  e^{\gamma^{2}t/2} \mathrm{Erf}\left(\gamma\sqrt{t/2}\right)\right).
\end{equation*}
\end{proof}

\section{Concluding remarks}
In this paper, we introduce and study the GSP and a fractional version of it, namely, the GFSP. The GSP and GFSP are Skellam type variants of the GCP and GFCP, respectively. Some distributional properties such as the pmf, pgf, mean, variance and covariance are derived for these processes. It is shown that the GSP and GFSP exhibits the LRD property. We obtain the systems of differential equations that govern their state probabilities. Two time-changed versions of the GFCP, namely, TCGFCP-I and TCGFCP-II are considered by time-changing it by an independent L\'evy subordinator and its inverse. We obtain a version of the law of iterated logarithm for the TCGFCP-I. Some particular cases of these time-changed processes are considered by choosing specific L\'evy subordinators such as the gamma subordinator, the TSS, the IGS and their inverse. For these particular cases, we obtain the governing system of differential equations for their state probabilities. It is known that the GCP has application in risk theory (see Kataria and Khandakar (2021c)). We expect the TCGCP-I to have potential application in risk theory as it exhibit the LRD property.

\end{document}